\documentclass[10pt]{article}
\usepackage{amsmath}
\usepackage{amsfonts}
\usepackage{amssymb}
\usepackage{amsthm}
\usepackage{mathtools}
\usepackage{algorithmic, algorithm}
\usepackage{graphicx}
\usepackage[auth-sc]{authblk}
\usepackage[utf8]{inputenc}
\usepackage[english]{babel}
\usepackage{geometry}
\usepackage{cite}
\usepackage{xcolor}
\usepackage{caption}
\newtheorem{theorem}{Theorem}[section]
\newtheorem{proposition}[theorem]{Proposition}

\newtheorem{lemma}[theorem]{Lemma}
\newtheorem{remark}[theorem]{Remark}
\newtheorem{assumption}[theorem]{Assumption}
\newtheorem{definition}[theorem]{Definition}
\usepackage{subcaption}

\newcommand{\R}{\mathbb R}
\newcommand{\N}{\mathbb N}

\newcommand{\norm}[1]{\left\| #1 \right\|}
\newcommand{\abs}[1]{\lvert#1\rvert}
\newcommand{\set}[1]{\left\{ #1\right\}}

\numberwithin{equation}{section}
\numberwithin{algorithm}{section}

\graphicspath{{./fig/}}

\begin{document}
\title{Discontinuous Galerkin methods for the Laplace-Beltrami operator on point clouds\footnote{Emails: guozhi.dong@csu.edu.cn; hailong.guo@unimelb.edu.au; zqshi@tsinghua.edu.cn}}
\author[1]{Guozhi Dong}
\author[2]{Hailong Guo}
\author[3,4]{Zuoqiang Shi}
{\tiny
\affil[1]{School of Mathematics and Statistics, HNP-LAMA, Central South University, Changsha 410083, China}
\affil[2]{School of Mathematics and Statistics, The University of Melbourne, Parkville, VIC, 3010, Australia}
\affil[3]{Yau Mathematical Sciences Center, Tsinghua University, Beijing, 100084, China}
\affil[4]{Yanqi Lake Beijing Institute of Mathematical Sciences and Applications, Beijing, 101408, China}}
\date{ }
\maketitle
\begin{abstract}
This paper is dedicated to the development of numerical analysis for high-order methods solving partial differential equations on scattered point clouds. We build a novel geometric error analysis framework by estimating the error in the approximation of the Riemann metric tensor. The innovative framework serves as a fundamental tool for analyzing discontinuous Galerkin methods applied to the Laplace-Beltrami operator on possibly discontinuous geometry. We provide numerical examples on patchy surfaces reconstructed from point clouds to support our theoretical findings.
	\vskip .3cm
	
\end{abstract}

%%%%%%%%%%%%%%%%%%%%%%%%%%%%%%%%%%%%%%%%%%%%%%%%%%%%%%%%%%%%%%%%%%%%%%
	
\section{Introduction}
Numerical solutions of partial differential equations (PDEs) on (unknown) manifolds or surfaces have garnered substantial attention in recent years,  with particular applications including surface evolution and geometric flows.
Since Dziuk \cite{Dziuk1988} introduced the surface finite element method (SFEM), substantial efforts and progress have been made
in this direction. Notable recent contributions and reviews can be found, for example, in \cite{DecDziEll05,DziukElliott2013,KovLiLub19,BONITO20201,HuLi22,BaiLi23,DuanLi24}.
Numerical methods for PDEs on manifolds typically encompass two levels of discretization (or approximation): one is the discretization of the geometry and its differential structures;  the other one is the discretization of the functions defined on the geometry.  The greater part of the literature on SFEM seems to have been based on the fact that the underlying geometry (manifolds) are surfaces represented by level set functions (e.g., sign distance functions). The geometry approximations are usually limited to continuous piecewise polynomial patches which interpolate underlying smooth manifolds. Using this idea, many of the well-known numerical methods in planar domains have been generalized to curved spaces or even evolving manifolds.
Exemplary works  include SFEM, surface finite volume method, surface Crouzeix-Raviart element and so on \cite{DziukElliott07,Demlow2009,DuJu2005,Guo20,OlshanskiiReuskenGrande2009,GraLehReu18}. 
Isogeometric analysis (IGA) is another way of integrating geometric information with numerical solution process for PDEs, which leverages Non-Uniform-Rational-Bsplines for both geometry representation and function approximation to enhance the solution accuracy and efficiency, particularly popular in computer aided design. There are also works to consider IGA for solving PDEs on surfaces, see for example \cite{BarDedQua15,PanRabYan21}.
Along these lines, discontinuous Galerkin (DG) methods with respect to continuously interpolated piecewise polynomial surfaces have been investigated as well \cite{DedMadSti13,AntDedMadStaStiVer15,DednerMadhavanStinner2016}.
	The vast majority of the geometric error estimations in the existing works are based on the fundamental results of \cite{Dziuk1988} for linear SFEM and \cite{Demlow2009} for higher-order ones, which are for interpolating type discretizations.
	In other words, the geometric errors are residuals of local polynomial interpolations of nodal points on the exact manifolds.
	In addition, most of the existing works calculate the surface gradient and other differential operators based on the explicit tangential projection from the ambient space onto the tangent spaces of the manifolds.
That means explicit embedding (e.g., through level set functions) of the manifolds to their ambient Euclidean spaces are assumed to be known.
	These assumptions are not necessarily satisfied for problems we study in this paper, i.e., manifolds are represented by point clouds.

 Due to the revolutionary development of digital technologies, point clouds have become a standard and convenient method for sampling surfaces across numerous real-world applications. For example, solving the Laplace-Beltrami eigenvalue problem on surfaces aids in discrete surface registration \cite{LaiZha17}. However, owing to the discrete nature of point clouds, the existing numerical methods for solving PDEs on point clouds are typically meshfree types, to name just a few  \cite{LiaZha13,LiShi16}. The former relies on local polynomial surface approximation, offering a direct and implementable approach but posing challenges for theoretical analysis. The latter, the point integral method proposed in \cite{LiShi16}, has demonstrated convergence, albeit with relatively low convergence rates.

In this paper, we aim to develop high-order numerical methods to solve PDEs on point clouds with solid theoretical analysis. To do this, a patchwise manifold needs to be reconstructed, which makes high-order numerical methods on point clouds possible. To achieve a high-order numerical method, we first need a high-order approximation of the underlying surface. Since the surface is given by unstructured point clouds, it is difficult to obtain a global high-order representation. However, it is relatively easy to compute local high-order approximations to the surface by polynomial fitting. However, in general, such local approximations may be inconsistent with each other, i.e., the local patches may be discontinuous. Based on the discontinuous geometric representation, the discontinuous Galerkin (DG) method becomes a natural choice. Another difficulty lies in the theoretical error analysis. The geometric assumptions from existing results are violated in this setting. Thus, it calls for new ideas in both algorithm designing and numerical analysis for such problems.

In recent works \cite{DonGuo20,DonGuoGuo24}, the idea of using the Riemannian metric tensor for post-processing and analyzing geometric errors has been employed in gradient recovery (post-processing of numerical solutions) and computation of tangential vector fields on manifolds using linear SFEM, which has proven successful. This motivates us to further develop the ideas presented in \cite{DonGuo20,DonGuoGuo24} for higher-order SFEMs and DG methods, as well as for {\it a priori} error analysis.
In this paper, we consider the Laplace-Beltrami equation and its corresponding eigenvalue problem as model equations. Different from the methods widely adopted in the literature, we take an intrinsic viewpoint. By directly estimating the approximation error of Riemannian metric tensors, the error analysis of numerical solutions, particularly for geometric errors, becomes more transparent in comparison with existing results.
Instead of focusing on an abstract unified DG framework as presented in \cite{AntDedMadStaStiVer15, ABCM2001}, we concentrate on the interior penalty DG method \cite{Arn82, DiEr2012, Rivi2008,HeWa2008} for solving these representative problems. It is worth noting that convergence rates of DG methods for eigenvalue problems on planar domains have been studied, for example in \cite{AntBufPer06}, and higher-order SFEMs for eigenvalue problems were proposed and analyzed in \cite{BonDemOwe18}.
Particularly observed in \cite{BonDemOwe18} was a non-synchronization phenomenon between geometric error and function approximation error in terms of eigenvalue convergence rates. Whether the error bounds presented in \cite{BonDemOwe18} are sharp or not remains an open question.
As a side product, this paper will overcome these restrictions and address these open issues.

%	\paragraph*{\textbf{Contribution and structure of the paper}}
The primary contributions are an algorithmic framework employing DG methods for solving elliptical PDEs on point clouds and a novel  geometric error analysis method. The approach involves approximating the Riemannian metric tensor of the underlying manifold, differing from existing methods mostly based on the works in \cite{Dziuk1988,Demlow2009}, which rely on the tangential projection for defining differential operators on surfaces. In our work, these differential operators are represented using local geometric parameterizations and their induced metric tensors. The tangential projection and the explicit embedding maps of the manifolds, which are required in existing works in the literature, e.g., \cite{Dziuk1988,Demlow2009}, are not compulsory to know using our new framework. Moreover, it even does not require the surface patches to be continuous, therefore is capable of dealing with more general geometry and allows more flexibility in geometric approximation and numerical methods. This also serves as a starting point for us to study DG methods on patchwise manifolds. 
More explicitly, the error estimate of the metric tensor is established in Theorem \ref{thm:geo_error}, where a discrepancy phenomenon between the Jacobian approximation and the metric approximation is reported. In Remark \ref{rem:discrepancy}, an example is provided to justify that our geometric estimates are crucial. These results guide the compatibility of the geometric discretization with the function approximation to have optimal convergence rates in numerical methods. Using the intrinsic geometric representation, we can design flexible numerical methods and conduct novel error analysis for problems on curved domains. In particular, using the new framework, we prove the convergence rates for high-order DG methods without global continuity of the surface patches in Theorem \ref{thm:convergence}. 

In the numerical analysis of the eigenvalue problem, our approach is also different from both the two works in \cite{AntBufPer06,BonDemOwe18}. We show that eigenvalues are invariant under geometric transformations. Then we use the Babu\v{s}ka-Osborn \cite{BabOsb91} framework and apply it to DG methods on discretized manifolds in Theorem \ref{thm:eigen_result}, which gives the same theoretical rates as in \cite{BonDemOwe18} for high-order SFEMs. Although we focus only on the interior penalty DG (IPDG) method in the analysis and the numerical examples, our framework and ideas are applicable to other DG methods and SFEMs as well. Last but not least, using a simple example (see Remark \ref{rem:eigen_sharp}) we argue that the theoretical rates proved in the paper for the convergence of the eigenvalue (therefore also the results in \cite{BonDemOwe18}) is theoretically optimal.

The rest of the paper is organized as follows:
Section \ref{sec:geometry} presents a general description of our geometric setting, including an algorithm for reconstructing patchwise manifolds from point clouds and the geometric error analysis.
Section \ref{sec:Laplace} analyzes the convergence rates of the numerical solution of the IPDG method for the Laplace-Beltrami equation and its eigenvalue problem on the patchwise manifolds.
Finally, the paper is concluded after numerical results and discussions in Section \ref{sec:numerical}.
Finally, in Section \ref{sec:con}, we draw conclusions. A few technical details are deferred in Appendices \ref{app:approx} and \ref{app:DG_bilinear}.

\section{Patches reconstructed from point clouds and their error analysis}
	\label{sec:geometry}

We consider a smooth (sub-)manifold denoted by $\Gamma$ associated with a Riemannian metric tensor $g$. Please note that the term 'smooth' will be precisely defined later. To provide specificity in our discussion, we choose two-dimensional surfaces $\Gamma \subset \mathbb{R}^3$ as examples in this paper.
Given a point cloud $\Sigma=\{\xi_i\}_{i\in I}$ sampled from $\Gamma$, in order to apply DG methods and their analysis for solving PDEs, we need to first reconstruct a patchwise manifold. This process involves two steps: (i) establishing a reference mesh $\Omega_{h}$ and (ii) performing the patch reconstruction.
In the subsequent sections, we will first provide a general description of the patchwise manifolds required for this work. Following that, we will introduce an algorithm for reconstructing these types of patches from point clouds. Lastly,  we will present the corresponding geometric error analysis.

	\subsection{Geometry and its patchwise approximation}
Let $\Gamma_{h}:=\bigcup_{j\in J} \Gamma_{h}^j$ represent a subdivision of $\Gamma$, where $h$ denotes a parameter determined by the subdivision's size, and $J\subset \N$ stands for an index set. Each $\Gamma_h^j\subset \Gamma$ is homeomorphic to an open subset in $\R^2$. Specifically, we assume $\Gamma = \bigcup_{j\in J} \overline{\Gamma_{h}^j}$ and $\Gamma_h^{j_1}\cap \Gamma_h^{j_2}=\emptyset$ for all index pairs $j_1\neq j_2\in J$.
In practical applications, the explicit representation of the manifold $\Gamma$ might be unavailable. Thus, $\Gamma_{h}$ remains unknown, and only approximations or discrete representations are accessible. Consequently, we consider a collection of patches $\hat{\Gamma}_h=\bigcup_{j\in J}\hat{\Gamma}^{j}_h$, where each $\hat{\Gamma}^{j}_h$ approximates $\Gamma_h^j$ for every $j\in J$. We assume that every $\hat{\Gamma}^{j}_h$ is also homeomorphic to an open subset in $\R^2$, and there exists a bijective map between $\hat{\Gamma}^{j}_h$ and $\Gamma^{j}_h$. Particularly, we consider the following $k$-th order patchwise manifold:

\begin{definition}[$k$-th order patchwise manifold]
Let $\hat{\Gamma}^{k}_h=\bigcup_{j\in J} \hat{\Gamma}^{k,j}_h$ and $\Gamma_{h}=\bigcup_{j\in J} \Gamma_{h}^j$ be collections of patches. We define $\hat{\Gamma}^{k}_h$ as a $k$-th order patchwise manifold approximating $\Gamma_{h}$ if a family of bijective maps exists between $\Gamma_{h}^j$ and $\hat{\Gamma}^{k,j}_h$ for every $j\in J$. Particularly, these maps satisfy the following estimate:
\begin{equation}\label{eq:patch_dist}
\norm{g-g_h^k}_{L^\infty}\leq C h^{k+1},
\end{equation}
where $g$ and $g_h^k$ denote the Riemannian metric tensors associated with the patches in $\Gamma_h$ and $\hat{\Gamma}^{k}_h$, respectively. The $L^\infty$ norm is evaluated on the pullback of the metric tensors to every common homeomorphic open set for the corresponding patch pairs $\hat{\Gamma}^{j}_h$ and $\Gamma^{j}_h$.
\end{definition}

Throughout this article, the letter $C$ or $c$, with or without subscripts, denotes a generic constant which is independent of $h$ and may not be the same at each occurrence. For convenience, we represent $x\leq Cy$ (or $x\geq Cy$) as $x\lesssim y$ (or $x\gtrsim y$).
Later in this section, we shall provide the detail about how to construct $\hat{\Gamma}^{k}_h$ from point clouds.
It is important to note that for patchwise manifolds $\{\hat{\Gamma}^{k}_h\}$, global continuity is not always necessary. An example of such manifolds includes local polynomial patches. Algorithm \ref{alg:point_estimate} illustrates how to construct them from a given point cloud representation of $\Gamma$. 
To achieve this, we introduce a parameter domain denoted by $\Omega_{h}=\bigcup_{j\in J} \Omega_h^j$. We shall abuse a bit of notation in the paper, as later $\Omega_h^j$ may denote both a planar domain (e.g., a triangle) and also the corresponding hyperplane to which the parameter domain belongs.
Subsequently, we establish a patchwise mapping $\pi:\Omega_h \to \Gamma_h$. One approach to defining this map is by utilizing normal vectors on $\Gamma_h$ as in \eqref{eq:projection}. 
On each $\Omega_{h}^j$, we construct a local coordinate system where the origin lies within $\Omega_h^j$, and $\pi$ serves as a mapping from $\R^2$ to $\R^3$. This mapping $\pi$ induces a metric tensor $g$ of $\Gamma$ through the following relation:
\begin{equation}\label{eq:metric_tensor}
g\circ \pi =(\partial \pi)^\top \partial \pi .
\end{equation}
It is important to highlight that $\pi$ isn't necessarily defined using the global coordinates in $\R^3$, but alternatively defined on each local patch $\Omega_h^j$ using the local coordinates within it. Here, $\partial$ denotes the Jacobian, for instance, $\partial \pi\in \R^{3\times 2}$ signifies the Jacobian of the geometry map evaluated locally based on the local coordinates on $\Omega_{h}^j$.
Remarkably, \eqref{eq:metric_tensor} plays a pivotal role in the error analysis presented in this paper.

A bijective map $\pi^k_h:\Omega_h \to \hat{\Gamma}^{k}_h$ can be established analogously to $\pi$. Consequently, we obtain the bijective map $\pi^k_h\circ \pi^{-1}:\Gamma_h \to \hat{\Gamma}^k_h$.
Using the parametrization map $\pi$ (or $\pi^k_h$), the gradient of a scalar function on manifolds can be computed as follows:
\begin{equation}\label{eq:gradient}
(\nabla_g v)\circ \pi= \partial \pi (g^{-1}\circ\pi) \nabla \bar{v} \quad \text{ or } \quad 	(\nabla_{g^k_h} v^k_h)\circ \pi^k_h= \partial \pi^k_h ((g^k_h)^{-1}\circ\pi^k_h) \nabla \bar{v}^k_h. 
\end{equation}
Here, $\bar{v}=v\circ \pi$ (or $\bar{v}^k_h=v^k_h\circ \pi^k_h$), and $\nabla$ represents the gradient operator on the planar parametric domain $\Omega_h$. Note that this expression of surface gradient is different from the one used in many references for surface finite element methods like \cite{Demlow2009, Dziuk1988}. This novel approach enables the development of a new numerical analysis framework to investigate PDEs on surfaces or (sub-)manifolds more broadly.

One should be aware that the choice of $\Omega_{h}$ is flexible, and the mapping $\pi:\Omega_h\to \Gamma_h$ is non-unique. Though there are many different choices of $\pi$, the metric tensor $g$ is invariant. 
Without loss of generality, we consider in this paper $\Omega_{h}$ consisting of triangle faces of some polyhedron close to $\Gamma_{h}$, and the following particular widely used mapping $\pi:\Omega_{h} \to \Gamma_h$ in \eqref{eq:projection}: 
\begin{equation}\label{eq:projection}
\pi(x)=x-d(x)\nu(x)\quad \text{ for every }\; x\in \Omega_h \;\text{ in the local coordinates on } \Omega_h,
\end{equation}
where $\nu(x)=\nu(\pi(x))$ is the  unit outward normal vector at the point $\pi(x)$ on $\Gamma$, and $d(x)$ is the sign distance function of $\Gamma$.
Note that the expression in \eqref{eq:projection} is coordinate independent. We always consider a patchwise local coordinate system where the origin is on the face of $\Omega_{h}$.
Assuming $\Omega_h$ to be close enough to $\Gamma$, then the mapping $\pi$ is bijective between $\Omega_h^j$ and $\Gamma_h^j \subset \Gamma$. This can be argued as follows:
The Jacobian of $\pi$ is $\partial \pi(x)=P_\tau -d(x) H$. Here $H$ is the second fundamental form, whose eigenvalues are the principal curvatures $\kappa$ of $\Gamma$. $P_\tau$ is the tangential projection and it satisfies $P_\tau \partial \pi(x)=\partial \pi(x)=\partial \pi(x) P_\tau $. 
Let $\kappa$ be the scalar curvature of $\Gamma$. It is easy to see that  the rank of $\partial \pi$ is the same as the dimension of $\Gamma$ provided that $d\kappa<1$. Thus the pseudo-inverse of $\partial \pi$ is bounded, and we have $\pi:\Omega_h \to \Gamma_h$ bijective. On the other hand $d\kappa<1$ requires the distance of $\Omega_h$ to $\Gamma_h$ to be inversely proportional to the curvature.

Eventually, we will develop DG methods for solving PDEs on $\hat{\Gamma}^{k}_h$ instead of $\Gamma_{h}$ with a novel error analysis based on local parametrization.
Although we have introduced the concept of \emph{$k$-th order patchwise manifolds}, it remains unclear how to obtain such patchwise manifolds from point clouds and under which conditions to have the approximation property in \eqref{eq:patch_dist}.
In the rest of this section, we shall address this problem, which is important for designing optimally convergent high-order DG methods on point clouds.

	\subsection{Reference mesh on point clouds}
To get the reference mesh, we reconstruct an initial triangular mesh
	$\Omega_H$. For this initial mesh, we only require that it is sufficiently close to the surface such that the map $\pi$ given in \eqref{eq:projection} is well defined. This requirement is very mild. Generating the initial mesh can be easily achieved from a point cloud using surface reconstruction methods such as screened Poisson reconstruction (SPR) \cite{SPR}, Gauss reconstruction \cite{GR}, and others. Figure \ref{fig:mesh-ref} illustrates an example of the initial triangular mesh constructed using SPR.
	\begin{figure}[!h]
		\centering
		{\includegraphics[width=0.8\textwidth]{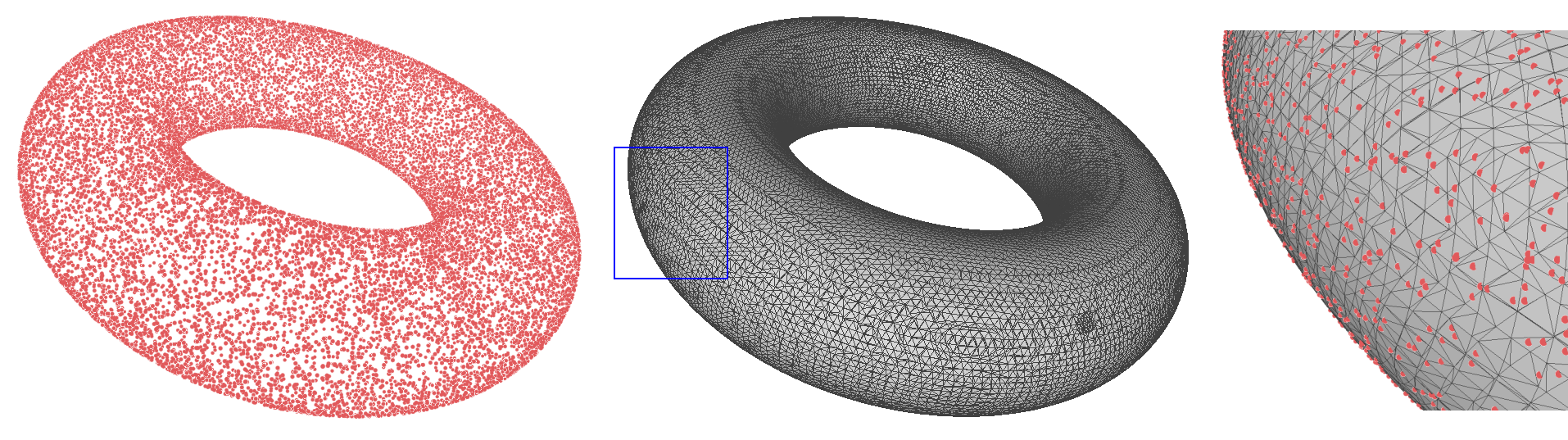}}
		\caption{(Color online) Initial meshes constructed from point cloud. Left: Point cloud; Middle: Reference mesh reconstructed by SPR \cite{SPR}; Right: Zoom-in of the point cloud and triangular meshes.}
		\label{fig:mesh-ref}
	\end{figure}

 If the initial mesh is too coarse to be used in the computation, we can refine the initial mesh to a satisfactory one. Each triangle in the mesh is subdivided into four triangles by connecting its edge centers until the diameter of the largest triangle reaches the filling distance $h_s$ outlined in Definition \ref{def:fill_dis}. 
	\begin{definition}\label{def:fill_dis}
		The filling distance $h_s$ of the point cloud $\Sigma$ is 
		\begin{equation}\label{eq:fill_dis}
		h_s :=\sup_{\psi\in \Gamma} \min_{\xi\in \Sigma } \abs{ \psi -\xi}.
		\end{equation}
	\end{definition}

 Note that though the definition of filling distance involves the knowledge of $\Gamma$, we only need a magnitude estimate of it for the convergence analysis. This could be achievable without knowing $\Gamma$; instead, we may use the largest distance of neighbouring points in $\Sigma$, which is close to $2h_s$, to have such a bound.

The refined mesh is denoted as $\tilde{\Omega}_h$. Subsequently, for each vertex of $\tilde{\Omega}_h$, we identify its $k$-nearest neighbors ($k\ge 3$) in the point cloud $\Sigma$ to fit a hyperplane. The vertex is then projected onto the hyperplane to obtain a new vertex. Finally, connecting all the new vertices yields the reference mesh $\Omega_h$, maintaining the same connections as those in $\tilde{\Omega}_h$. 
It can be easily verified that $\Omega_h$ possesses the property:
\begin{equation}\label{eq:dist_ref}
\text{dist}(\Omega_h, \Gamma_h)\sim h_s^2,
\end{equation}
This property will be utilized in the error analysis to achieve an optimal convergence rate.
Throughout the following discussion, we assume that the filling distance $h_s$ in \eqref{eq:fill_dis} is sufficiently small, and the largest diameter of the triangles on $\Omega_h$ satisfies $c_a h_s \leq h\leq  c_b h_s$. Here, $c_a$ and $c_b$ are constants, controlling the magnitude of $h$ to ensure that the polynomial reconstruction in Algorithm \ref{alg:point_estimate} is well-posed, meaning that the polynomial function is uniquely determined from the scattered points there.
	\subsection{Patch reconstruction from point cloud}
	\label{subsec:geo_algo}

Using the previously obtained reference mesh as a foundation, we are able to create local patch approximations by Algorithm \ref{alg:point_estimate}.
	\renewcommand{\thealgorithm}{\arabic{algorithm}}
	\setcounter{algorithm}{0}
	\begin{algorithm} [!h]
		\caption{Patch reconstruction from point cloud}
		\textbf{Input}: The given point cloud $\Sigma$, and the reference mesh $\Omega_{h}$.
		
		Let $J$ be the total number of small triangles on $\Omega_{h}$. For every $j\in J$, implement the following steps in order.
		\begin{itemize}
			\item[(1)] On each reference element $\Omega_{h}^j$, choose the barycenter $c^j$, pick a set of $l$ nodal points, and select the closest points in $\Sigma$ to each one of the $l$ nodal point on $\Omega_{h}^j$, denoted by $\xi_1^j,\; \xi_2^j,\; \ldots,\; \xi_L^j$ for some $L\geq l$. 
			\item[(2)] Choose a local coordinate system where the hyperplane of $\Omega_{h}^j$ is the transverse plane, and let $c^j$ be the origin. Then we change the coordinates of $\set{\xi_i^j}_{i=1}^l$ according to the new local system. In the local coordinates, $\set{\xi_i^j}_{i=1}^l$ is represented as $(s_i^j,v_i^j)$ for $s_i^j\in \Omega_{h}^j$. 
			\item[(3)] Then approximate (interpolate) the points using a $k^{th}$ order polynomial function graph defined on the hyperplane of $\Omega_{h}^j$. That is to find $p_h^{k,j}$ such that
			\begin{equation} \label{eq:local_patch_appro}
			p_h^{k,j}=\underset{p}{\operatorname{argmin}} \frac{1}{L}\sum_{i=1}^L \abs{p(s_i^j) - v_i^j}^2 \text{  over } p\in \mathbb{P}^k(\Omega_{h}^j). 
			\end{equation}
			
			\item[(4)] Select some interpolating points $(x_k^j)_k\subset \Omega_{h}^j$ in the parameter domain. Typically, those points $x_k^j$ are chosen according to the polynomial degree.
			\item[(5)] Compute $\psi_k^j=\pi_h^k(x_k^j)$ for those selected points $x_k^j\in \Omega_{h}^j$ using the iterations described in \eqref{eq:Newton_ite}.
			\item[(6)] Interpolate $\psi_k^j$ component-wise to get the corresponding polynomial patch on $\Omega_{h}^j$ denoted as $\hat{\Gamma}_h^{k,j}$.
		\end{itemize}
	\textbf{Output}:  The surface patches $(\hat{\Gamma}_h^{k,j})_{j\in J}$ which approximates the underlying smooth surface $\Gamma$.
		\label{alg:point_estimate}
	\end{algorithm}
\begin{remark}
In the first step of Algorithm \ref{alg:point_estimate}, the nodal points in $\Omega_h^j$ are chosen as the interpolation points of $P_k$ element \cite{Ci2002}. The number of nodal points $l$ then depends on the degree of the polynomial function $p$ in \eqref{eq:local_patch_appro}. 
 Problem \eqref{eq:local_patch_appro} is a least square problem of an over-determined system and can be surely well-conditioned.  Note that $L\geq l$ is due to the reason that one can pick up multiple closest points in $\Sigma$ for each nodal points in $\Omega_h^j$.
  \end{remark}	
	
	%%%%%%%%%%%%%%%%%%%%%%%%%%%%%%%%%%%%%%%%%%%%%%%%%%%%%%%%%%%%%%%%%%%

We point out that $\hat{\Gamma}_h^{k,j}$ is a piecewise polynomial patch, but not necessarily interpolating $\Gamma_h$.	
Now we provide the details on computing $\psi^j_k$ in Step (5) Algorithm \ref{alg:point_estimate}. The idea is inspired from the equation in \eqref{eq:projection}. More explicitly, given a point $x\in \Omega_h^j$ over the parametric domain, we compute the point $\psi \in \hat{\Gamma}_h^{k,j}$ such that $\psi(x) - x =- d_h^k(x)\nu_h^{k,j}(x)$, where $\nu_h^{k,j}$ denotes the outward unit normal vector of the local polynomial graph, and $d_h^k$ is the length of the vector $\psi -x$. The point $\psi(x)$ is then analogous to $\pi(x)$.
Let $\psi=(\psi_1,\psi_2,\psi_3)$ be a local polynomial graph, i.e., $\psi_3=p_h^{k,j}(\psi_1,\psi_2)$, and we consider the local coordinates, $x=(x_1,x_2,0)\in \Omega_h^j$. Notice that
$d_h^k(x)=\sqrt{(\psi_1-x_1)^2+(\psi_2-x_2)^2+(\psi_3)^2}$, and $\nu_h^{k,j}(x)=\frac{(-\partial_{1} p_h^{k,j},\; -\partial_{2} p_h^{k,j}\; ,1)^\top }{\sqrt{(\partial_{1} p_h^{k,j})^2 +(\partial_{2} p_h^{k,j})^2+1}}$.

	This gives us the following nonlinear system
	\begin{equation}\label{eq:system1}
	\begin{aligned}
	\psi_3 \partial_{1} p_h^{k,j}(\psi_1,\psi_2)+\psi_1-x_1=0,\\
	\psi_3 \partial_{2} p_h^{k,j}(\psi_1,\psi_2)+\psi_2-x_2=0,\\
	\psi_3-p_h^{k,j}(\psi_1,\psi_2)=0.
	\end{aligned}
	\end{equation}
	We rewrite the above system in a compact form as $M(\psi)=0$, which can then be solved using Newton's method.
	For that, we reduce the above system to a two-variable equation
	\begin{equation}\label{eq:system2}
	\begin{aligned}
	M_1(\psi_1,\psi_2)=p_h^{k,j}(\psi_1,\psi_2) \partial_{1} p_h^{k,j}(\psi_1,\psi_2)+\psi_1-x_1=0,\\
	M_2(\psi_1,\psi_2)=p_h^{k,j}(\psi_1,\psi_2) \partial_{2} p_h^{k,j}(\psi_1,\psi_2)+\psi_2-x_2=0.
	\end{aligned}
	\end{equation}
	This system then gives rise to the following iterations for some initial guess, e.g., $(\psi_{1,0},\psi_{2,0})^\top=(x_1,x_2)^\top$:
	\begin{equation}\label{eq:Newton_ite}
	(\psi_{1,{m +1}},\psi_{2,{m +1}})^\top =(\psi_{1,{m}},\psi_{2,{m}})^\top - (\nabla M^m)^{-1}(M^m_1,M^m_2)^\top; \; \psi_{3,{m+1}}=p_h^{k,j}(\psi_{1,m+1},\psi_{2,m+1}),
	\end{equation}
	where $M^m_i=M_i(\psi_{1,m},\psi_{2,m})$, and 
	\[\nabla M^m =\left( \begin{matrix}
	\partial_1M_1(\psi_{1,m},\psi_{2,m}) & \partial_1 M_2(\psi_{1,m},\psi_{2,m})\\
	\partial_2M_1(\psi_{1,m},\psi_{2,m}) & \partial_2 M_2(\psi_{1,m},\psi_{2,m})
	\end{matrix}\right),\]
	 which is a symmetric matrix. Note that its inverse matrix can be explicitly calculated
	\[ (\nabla M^m)^{-1} = \frac{1}{\det(\nabla M^m) }\left( \begin{matrix}
	\partial_2 M_2(\psi_{1,m},\psi_{2,m})& - \partial_1M_2(\psi_{1,m},\psi_{2,m})\\
	- \partial_2 M_1(\psi_{1,m},\psi_{2,m}) & \partial_1M_1(\psi_{1,m},\psi_{2,m})
	\end{matrix}\right) .\] 
	The Newton iterations are terminated once an expected accuracy level is reached (in our examples, we choose it to be  $10^{-14}$).
	\begin{remark}
		We notice here that Steps $(4)$ and $(5)$ in Algorithm \ref{alg:point_estimate} are important in order to prove Proposition \ref{prop:nodal_accu}, particularly the estimate in \eqref{eq:geo_approx2}.
	\end{remark}

One may notice that a continuous surface could potentially be reconstructed from the above algorithm, where one can add an additional step after Step (5) by averaging the nodal points corresponding to the same edges in the reference mesh. This would unify the nodal points on the common edge, and  then use the same polynomial interpolation as in Step (6)  of Algorithm \ref{alg:point_estimate} to achieve a continuous patchwise reconstruction. The same error bounds will hold for this continuous approximation, though proving it will require a bit more work. Since our analysis can cover both discontinuous and continuous cases, and the former appears to be more general in the error analysis, we will focus on the discontinuous case in this paper, as high-order numerical methods for PDEs on continuous surfaces have already been discussed in the literature, e.g., in \cite{Demlow2009}.

We now demonstrate the approximation property for the nodal points $\{\psi^i\}_{i\in\N}$ reconstructed from Algorithm \ref{alg:point_estimate}.
Considering the equation in \eqref{eq:system1}, the partial differential operators might seem to reduce the approximation accuracy of the polynomial functions $p_h^j$ after the Newton iterations. However, in the following, we show that surprisingly this is \emph{not} the case, as long as the Newton algorithm \eqref{eq:Newton_ite} converges sufficiently.

We employ the subsequent vector decomposition in the paper.
\begin{definition}\label{def:decomposition}
	On every local plane $\Omega_h$, we denote $\xi_{\Omega_h}$ and $\xi_{o}$ as the decomposed vectors of $\xi$, satisfying
	\[\xi =\xi_{\Omega_h} + \xi_o,\quad  \xi_o \perp  \Omega_h \quad \text{ and }\quad \xi_{\Omega_h} \perp \xi_o.\]
\end{definition}
Additionally, we introduce the following assumptions concerning $\Gamma_h$, $\Omega_{h}$, and $\hat{\Gamma}_h^k$, based on which we establish that the patches reconstructed from Algorithm \ref{alg:point_estimate} constitute a $k$-th order patchwise manifold.
\begin{assumption}\label{assum:geometry} 
	\begin{enumerate} 
		\item 	 $\Gamma$ is a $C^{k_0}$ smooth manifold with bounded curvature, and $k_0\geq \max\set{k,2}$, where $k$ is the polynomial order of the surface reconstruction in Algorithm \ref{alg:point_estimate}.
		\item Every triangular parameter domain $\Omega_h^j$ is $\mathcal{O}(h^2)$ close to the corresponding patch $\Gamma_h^j$, and shape regular. This is understood in the sense that for every selected triangle pair $\Omega_h^j$ and $\Gamma_h^j$, we have $\abs{\pi(x)-x}\lesssim h^2$ with some uniform constant $C$ independent of $h$ and $x$.
	\end{enumerate}
\end{assumption}

	\begin{proposition}\label{prop:nodal_accu}
		Suppose the Newton iteration in Algorithm \ref{alg:point_estimate} stops with a sufficiently small error. 
		Then the reconstructed nodal point $\{\psi^i=\pi_h^k(x^i)\}_{i\in\N} \subset \hat{\Gamma}_h^{k,j}$ and the precise nodal points $\{\xi^i=\pi(x^i)\}_{i\in\N}\subset \Gamma$ satisfy the following error estimate
		\begin{equation}\label{eq:geo_approx1}
		\abs{\psi^i-\xi^i }\lesssim  h^{k+1}.
		\end{equation}
		In particular, 
		\begin{equation}\label{eq:geo_approx2}
		\abs{\psi^i_{\Omega_h}-\xi^i_{\Omega_h}}\lesssim h^{k+2},
		\end{equation}
		where $\psi^i_{\Omega_h}$ and $\xi^i_{\Omega_h}$ are the nodal points of $\psi^i$ and $\xi^i$ projected onto the parametric domain $\Omega_h$, respectively.
	\end{proposition}

Before proving Proposition \ref{prop:nodal_accu}, we present an auxiliary lemma regarding the polynomial approximation of scattering points. This lemma constitutes a standard result available in various references, such as \cite{Wen05}. For the sake of completeness, a proof is included in Appendix \ref{app:approx}.
\begin{lemma}\label{lem:aux_appr}
	Let the local patches be reconstructed as outlined in Algorithm \ref{alg:point_estimate}, representing function graphs of $k^{th}$ order polynomials, and let Assumption \ref{assum:geometry} hold for the point cloud. Assuming that the selected points are chosen such that the polynomial reconstruction is well-posed for every patch in Algorithm \ref{alg:point_estimate}, then the fitting polynomial function satisfies
	\begin{equation}\label{eq:polynomial}
		\norm{p_h^{k,j}- p^j}_{L^\infty(\Omega_h^j)} \lesssim  h^{k+1} \; \text{ and } \; \norm{\nabla p_h^{k,j}- \nabla p^j}_{L^\infty(\Omega_h^j)} \lesssim  h^{k} \quad \text{ for all } \quad  j\in J,
	\end{equation}
	where $p_h^{k,j}$ and $p^j$ denote the local polynomial function and the precise function, respectively, with the graph providing the exact patch on $\Gamma$.
\end{lemma}

	\begin{proof}[\textbf{Proof of Proposition \ref{prop:nodal_accu}}] 
			In the proof, to ease the notation, we ignore all the index $j$, but keep in mind that we always consider the estimate patchwise.
		Recall the formula \eqref{eq:system1}. Then we have
		\begin{equation}\label{eq:decomposed_error}
		\psi^i_{\Omega_h}-\xi^i_{\Omega_h}= p( \xi^i_{\Omega_h}) \nabla p( \xi^i_{\Omega_h})-  p_h^{k}( \psi^i_{\Omega_h})\nabla p_h^{k}( \psi^i_{\Omega_h})\quad \text{ and } \quad  \psi^i_o- \xi^i_o= p_h^{k}( \psi^i_{\Omega_h})- p( \xi^i_{\Omega_h}),
		\end{equation}
		where $p_h^{k}:\Omega_h\to \R$ is the reconstructed local polynomial patch in Algorithm \ref{alg:point_estimate}, and $p:\Omega_h\to \R$ is the counterpart of $p_h^{k}$ associated to the patch on $\Gamma$. Note that $p$ and $p_h^{k}$ satisfy the estimate from Lemma \ref{lem:aux_appr}. In particular $k\geq 1$, referring to \eqref{eq:dist_ref} there are the estimates 
		\begin{equation}\label{eq:magnitude_est}
		\norm{p}_{L^\infty(\Omega_h)}\lesssim h^2 \quad \text{ and } \quad \norm{p_h^{k}}_{L^\infty(\Omega_h)}\lesssim h^2.
		\end{equation}
		To further simply the presentation, we will ignore the $L^\infty(\Omega_h)$ norm subscript in the rest of the proof.
		Note that \eqref{eq:magnitude_est} and \eqref{eq:polynomial} together imply that both $\norm{\nabla p}$ and $\norm{\nabla p_h^{k}} $ are of order $\mathcal{O}(h)$.
		Using the triangle inequality, the left equality in \eqref{eq:decomposed_error} leads to
		\begin{align*}
		|\psi^i_{\Omega_h}-\xi^i_{\Omega_h}|
		\leq &\norm{\nabla p( \xi^i_{\Omega_h})(p( \xi^i_{\Omega_h})-  p_h^{k}( \xi^i_{\Omega_h}) )}+ \norm{ (p_h^{k}( \xi^i_{\Omega_h})-  p_h^{k}( \psi^i_{\Omega_h}))\nabla p( \xi^i_{\Omega_h})}\\
		&	 +\norm{p_h^{k}( \psi^i_{\Omega_h})(\nabla p( \xi^i_{\Omega_h}) -   \nabla p_h^{k}( \xi^i_{\Omega_h})) }+\norm{ p_h^{k}( \psi^i_{\Omega_h})(\nabla p_h^{k}( \xi^i_{\Omega_h}) -  \nabla p_h^{k}( \psi^i_{\Omega_h}))}\\
		\leq 	&\norm{\nabla p( \xi^i_{\Omega_h})}\norm{p( \xi^i_{\Omega_h})-  p_h^{k}( \xi^i_{\Omega_h})}+
		L\norm{\psi^i_{\Omega_h}-\xi^i_{\Omega_h}}\norm{\nabla p( \xi^i_{\Omega_h})}\\
		&	 + \norm{ p_h^{k}( \psi^i_{\Omega_h})}\norm{\nabla p( \xi^i_{\Omega_h}) -   \nabla p_h^{k}( \xi^i_{\Omega_h})} +L'\norm{\psi^i_{\Omega_h}-\xi^i_{\Omega_h}}\norm{p_h^{k}( \psi^i_{\Omega_h})},
		\end{align*}
		where $L$ and $L'$ are the local Lipschitz constants of the polynomial functions $p_h^k$ and $\nabla p_h^{k}$ over $\Omega_h$, respectively. Particularly the magnitudes of $L$ and $L'$ are of order $\mathcal{O}(h)$ and $\mathcal{O}(1)$, respectively, thus uniformly bounded.
		Rearranging the above inequality, we have
		\begin{equation*}
		\begin{aligned}
		&\left(1- L\norm{\nabla p( \xi^i_{\Omega_h})} -L'\norm{p_h^{k}( \psi^i_{\Omega_h})}\right)	\norm{\psi^i_{\Omega_h}-\xi^i_{\Omega_h}}\\
		\leq&\norm{\nabla p( \xi^i_{\Omega_h})}\norm{p( \xi^i_{\Omega_h})-  p_h^{k}( \xi^i_{\Omega_h})} + \norm{ p_h^{k}( \psi^i_{\Omega_h})}\norm{\nabla p( \xi^i_{\Omega_h}) -   \nabla p_h^{k}( \xi^i_{\Omega_h})}.
		\end{aligned}
		\end{equation*}
		The above inequality tells that for sufficiently small $ h_0$,  there exists a constant $C$ such that 
		\[0< C \leq \left(1- L\norm{\nabla p( \xi^i_{\Omega_h})} -L'\norm{p_h^{k}( \psi^i_{\Omega_h})}\right) \quad \text{  for all }  \quad h\leq h_0.\]
		Then we have
		\begin{align*}
		&|\psi^i_{\Omega_h}-\xi^i_{\Omega_h}| \\
		\leq &\frac{1}{C}\left(\norm{\nabla p( \xi^i_{\Omega_h})}\norm{p( \xi^i_{\Omega_h})-  p_h^{k}( \xi^i_{\Omega_h})} + \norm{ p_h^{k}( \psi^i_{\Omega_h})}\norm{\nabla p( \xi^i_{\Omega_h}) -   \nabla p_h^{k}( \xi^i_{\Omega_h})}  \right),
		%	=  \mathcal{O}(h^{k+2}),
		\end{align*}
		which together with \eqref{eq:polynomial} and \eqref{eq:magnitude_est} confirm \eqref{eq:geo_approx2}, i.e., $\norm{\psi^i_{\Omega_h}-\xi^i_{\Omega_h}}\leq Ch^{k+2}$.
		Back to the second part of \eqref{eq:decomposed_error}, similarly we have
		\begin{equation*}
		\begin{aligned}
		| \psi^i_o- \xi^i_o|\leq& \norm{ p_h^{k}( \psi^i_{\Omega_h})- p( \psi^i_{\Omega_h})} + \norm{ p( \psi^i_{\Omega_h})- p( \xi^i_{\Omega_h})}\leq Ch^{k+1}+L\norm{\psi^i_{\Omega_h}- \xi^i_{\Omega_h}}.
		\end{aligned}
		\end{equation*}
		This concludes the statement in \eqref{eq:geo_approx1} by noticing that
		\[ |\psi^i-\xi^i|\leq | \psi^i_o - \xi^i_o | + |\psi^i_{\Omega_h}- \xi^i_{\Omega_h}|.  \]
	\end{proof}
The error bounds \eqref{eq:geo_approx1} and \eqref{eq:geo_approx2} enable us to apply previous results in Lemma \ref{lem:inter_error} for the reconstructed patches from Algorithm \ref{alg:point_estimate}.
\emph{Note that for continuous interpolated patches, the error in Proposition \ref{prop:nodal_accu} vanishes.}
\begin{remark}
	The outcome of Proposition \ref{prop:nodal_accu} signifies that to achieve identical error estimates, it is not imperative for the sample points from the point cloud to lie precisely on the manifold. In essence, it suffices to ensure that the points reside within a $\mathcal{O}(h^{k+1})$ neighborhood, such that \eqref{eq:geo_approx1} and \eqref{eq:geo_approx2} can be fulfilled to derive the error bounds of the metric tensors in \eqref{eq:patch_dist}.
\end{remark}

	\subsection{Geometric error analysis}
The bounded curvature of $\Gamma$ implies another useful estimate $\abs{\nu_{\Omega_h}}\lesssim h$ as demonstrated in the following lemma.
\begin{lemma}\label{lem:normal_proj}
	Let Assumption \ref{assum:geometry} be fulfilled. Consider $\Omega_h=\cup_{j\in J} \Omega_h^j$ as the parametric domain where $\pi:\Omega_h\to \Gamma_h$ is defined as in \eqref{eq:projection}. If the diameter of every $\Omega_h^j$ is on the scale $h<1$, then for every unit normal vector of $\Gamma$, there exists a constant $C$ independent of $h$ such that
	\[ \abs{\nu_{\Omega_h}} \lesssim h,\]
	where $\nu_{\Omega_h}$ is the decomposition of $\nu$ as defined in Definition \ref{def:decomposition}.
\end{lemma}
\begin{proof}
	Consider two arbitrary unit normal vectors on the same patch $\Gamma_h^j$, denoted by $\nu(\xi^a)$ and $\nu(\xi^b)$. The uniformly bounded scalar curvature $\kappa$ implies
	\[ \sup_{\xi^a\neq \xi^b\in\Gamma_h^j} \frac{\abs{\nu(\xi^a) -\nu(\xi^b)}}{\abs{\xi^a-\xi^b}}\lesssim 1. \]
	Consider three vertices of the triangle $\Omega_h^j$ and find their corresponding vertices on $\Gamma_h^j$, forming another triangle in the plane $\hat{\Omega}_h^j$. Given the assumed condition that the triangle in $\Omega_h^j$ is $\mathcal{O}(h^2)$ close to $\Gamma_h^j$ and is shape regular, the two outward unit normal vectors of $\Omega_h^j$ and $\hat{\Omega}_h^j$, denoted by $z_h^j$ and $\hat{z}_h^j$, respectively, satisfy the error bounds
	\[ \abs{z_h^j-\hat{z}_h^j}\lesssim h.\]
	Select $\xi_j^b\in \Gamma_h^j$ such that the unit normal vector is parallel to $\hat{z}_h^j$. Note that $(z_h^j)_{\Omega_h}=0$. Then we have $\abs{\nu_{\Omega_h}(\xi_j^b)}\leq\abs{z_h^j-\hat{z}_h^j} \lesssim h $, and $\abs{\xi-\xi_j^b} \lesssim h$ for all $\xi \in \Gamma_h^j$. For every $j\in J$, we derive
	\begin{equation}\abs{\nu_{\Omega_h}(\xi)}\leq \abs{\nu_{\Omega_h}(\xi_j^b)} + \abs{\nu_{\Omega_h}(\xi) -\nu_{\Omega_h}(\xi_j^b)  }\lesssim h+ \abs{\nu(\xi)-\nu(\xi_j^b)}\lesssim h+ \abs{\xi-\xi^b}. 
	\end{equation}
\end{proof}
The scaling constant $C$  can be further optimized with respect to the scalar curvature $\kappa$, although it is not pursued here.
	
	\begin{lemma}\label{lem:inter_error}
		Let Assumption \ref{assum:geometry} be satisfied.
		Let $\pi:\Omega_h \to \Gamma_h$ and $\pi^k_h: \Omega_h \to \hat{\Gamma}_h^k$ be the mapping associated to the original patches on $\Gamma $ and reconstructed patches on $\hat{\Gamma}_h^k$ using Algorithm \ref{alg:point_estimate}, respectively.
		Then we have the following error bounds:
		\begin{equation}\label{eq:inter_error}
		\begin{aligned}
		\norm{\partial \pi - \partial \pi^k_h}_{L^\infty(\Omega_h)} \lesssim h^{k} , & \quad \norm{\left(\partial \pi - \partial \pi^k_h\right)_{\Omega_h}}_{L^\infty(\Omega_h)} \lesssim h^{k+1},\\
		\norm{\left( \partial \pi^k_h \right)_o }_{L^\infty(\Omega_h)} \lesssim h,	 & \quad	\text{ and } \quad  \norm{\left(\partial \pi \right)_o }_{L^\infty(\Omega_h)} \lesssim h.
		\end{aligned}  
		\end{equation} 
		
	\end{lemma}
	
	\begin{proof}
		Definition \ref{def:decomposition} and  the linearity of the Jacobian imply that 
		\[\partial \pi=(\partial \pi)_o + (\partial \pi)_{\Omega_h}
		\quad \text{ and } \quad 
		\partial \pi^k_h=(\partial \pi^k_h)_o + (\partial \pi^k_h)_{\Omega_h} .
		\]
		Let $\tilde{\pi}_h^k $ be the polynomial patch by interpolating the exact nodal points $\set{\xi^i}$. 
		Therefore we have that $\tilde{\pi}_h^k=\sum_{i=1}^m \xi^i \phi_i$ and $\pi_h^k=\sum_{i=1}^m \psi^i \phi_i$, where $m$ is the number of nodal points depending on the interpolating polynomial orders, and $\set{\phi_i}_{i=1}^m$ are the Lagrange polynomial bases. Based on  Assumption \ref{assum:geometry} and Proposition \ref{prop:nodal_accu}, we have that
		\[ \norm{ \partial\tilde{\pi}_h^k - \partial \pi_h^k}_{L^\infty(\Omega_h)}\lesssim h^k. \]
		To show the first estimate in \eqref{eq:inter_error}, we simply use the triangle inequality and the interpolation error estimates to derive
		\[\norm{ \partial \pi - \partial \pi_h^k}_{L^\infty(\Omega_h)}\leq \norm{  \partial \pi - \partial \tilde{\pi}_h^k}_{L^\infty(\Omega_h)} + \norm{\partial \tilde{\pi}_h^k - \partial \pi_h^k}_{L^\infty(\Omega_h)} \lesssim h^k.\]
		For the second estimate in \eqref{eq:inter_error}, notice that $\left(\partial \tilde{\pi}_h^k - \partial \pi^k_h\right)_{\Omega_h}=\sum_{i=1}^m( \psi^i_{\Omega_h}-\xi^i_{\Omega_h})\nabla\phi_i$.
		Then by the triangle inequality we have
		\begin{equation}\label{eq:decompose}
		\norm{ (\partial \pi - \partial \pi_h^k)_{\Omega_h}}_{L^\infty(\Omega_h)}\leq \norm{ (\partial \pi - \partial \tilde{\pi}_h^k)_{\Omega_h}}_{L^\infty(\Omega_h)} + \norm{(\partial \tilde{\pi}_h^k - \partial \pi_h^k)_{\Omega_h}}_{L^\infty(\Omega_h)}.
		\end{equation}
		Now we estimate the first term on the right hand side of \eqref{eq:decompose}.  Since $\tilde{\pi}_h^k$ is the local polynomial interpolation of $ \pi$ parametrised using the coordinates on $\Omega_h$, for every point $x\in \Omega_h$, there exits a unit normal vector $\nu$ on $\Gamma_h$, so that $ \pi - \tilde{\pi}_h^k= (\tilde{d}_h^k)\nu$.  To see this, we use the map by $r(z)=z-\tilde{d}_h^k\nu(z)$ for all $z\in \Gamma_h$, and  $r:\Gamma_h\to \tilde{\Gamma}_h^{k}$ is a bijection. Recall that $z=\pi(x)=x-d\nu(x)$, which shows that $\pi(x)-\tilde{\pi}_h^k(x)=\tilde{d}_h^{k}\nu(x)$ where $\nu(x)=\nu(z)$ denotes the unit normal vector on $\Gamma_h$.
Using  \cite[Proposition 2.3]{Demlow2009} we have $\abs{\tilde{d}_h^k}\lesssim h^{k+1}$  which is the distances of points from $\pi$ to $\tilde{\pi}_h^k$  along the normal vector $\nu$. 

		Lemma \ref{lem:normal_proj} says that on every local patch $\Omega_h^j$, $\abs{(\nu)_{\Omega_h}}\lesssim h$.
		The product rule and the triangle inequality imply
		\begin{align*}
		\norm{ (\partial \pi - \partial \tilde{\pi}_h^k)_{\Omega_h}}_{L^\infty(\Omega_h)}	= &\norm{\partial (\tilde{d}_h^k  \nu_{\Omega_h})}_{L^\infty(\Omega_h)}
		\lesssim &
		h\norm{ \nabla (\tilde{d}_h^k)}_{L^\infty(\Omega_h)} + h^{k+1}\norm{ (\partial\nu)_{\Omega_h}}_{L^\infty(\Omega_h)}  =\mathcal{O}(h^{k+1}). 
		\end{align*}
		Note here we have used  \cite[Proposition 2.3]{Demlow2009} for $\norm{ \nabla (\tilde{d}_h^k)}_{L^\infty(\Omega_h)}=\mathcal{O}(h^{k})$, and $(\partial\nu)_{\Omega_h}$ is  bounded by curvature.
		The second term on the right hand side of \eqref{eq:decompose} can be estimated from the inequality in \eqref{eq:geo_approx2}. These conclude the proof of the second item in \eqref{eq:inter_error}. 
		The third and the fourth estimates in \eqref{eq:inter_error} follow from the fact that, for every fixed index $i$, $\psi^i_o$ and $\xi^i_o$ are both of $\mathcal{O}(h^2)$ distance to their corresponding parametric domain $\Omega_h$.
	\end{proof}
	
	%\begin{remark}
	%The fact that $\Omega_h$ for parameterizing the two polynomial patches may not be coincide. That is, there exists $\Omega_{h_n}$ for the point cloud patch.
	%However, there exist a regular and bijective polynomial mapping $\Phi:\Omega_h\to \Omega_{h_n}$ between the triangles. 
	%In particular, by \eqref{eq:geo_approx2}, the gradient of the mapping will satisfy
	%\[\abs{ \partial \Phi -  \text{Id}}\leq Ch^{k+1},\]
	%due to the polynomial interpolation property. Therefore, the result of Lemma \ref{lem:inter_error} will not be changed when taking into account this effect.
	%\end{remark}
	
	To simplify notations, we will not distinguish metric tensors $g$ and $g^k_h$ with $g\circ \pi$ and $g^k_h\circ \pi^k_h$ in the following.
	In Theorem \ref{thm:geo_error}, we show that $\bigcup_{j\in J}\hat{\Gamma}_h^{k,j}$ given by polynomial interpolating the nodal points $(\psi_i^j)_{i\in I_j}$ is a $k$-th order patchwise manifold. 
	\begin{theorem}\label{thm:geo_error}
		Given the same condition as Lemma \ref{lem:inter_error}, we have
		\begin{equation}\label{eq:metric_error}
		\begin{aligned}
		\norm{g^{-1} (g -g^k_h )}_{L^\infty}\lesssim h^{k+1}, &\quad\norm{\frac{\sqrt{\abs{g}}-\sqrt{\abs{g^k_h} }}{\sqrt{\abs{g}}}}_{L^\infty}  \lesssim h^{k+1},
		\\ \norm{\frac{l_{g_h^k}-l_g}{l_g}}_{L^\infty} \lesssim h^{k+1}, \quad \text{ and }	&\quad \norm{g^{-1} (\partial \pi )^\top n - (g_h^k)^{-1} (\partial \pi^k_h )^\top n^k_h }_{L^\infty}\lesssim h^{k+1},
		\end{aligned}  
		\end{equation}
		where $l_{g_h^k}$ and $l_g$ are restriction of $g_h^k$ and $g$ on the edges of the curved triangles, and $n$ and $n_h^k$ are the outer unit conormal vectors of $\partial \Gamma_h$ and $\partial  \hat{\Gamma}^{k}_h$ (i.e., the boundary of the curved triangles), respectively.

	\end{theorem}
	\begin{proof}
		Recall that $g\circ \pi=(\partial \pi)^\top \partial \pi$, and $g_h^k\circ \pi_h^k=(\partial \pi_h^k)^\top \partial \pi_h^k$. Note that all the quantities in \eqref{eq:metric_error} are independent of the local coordinate system, therefore, we choose both $\pi$ and $\pi_h^k$ the function graphs on $\Omega_{h}$. In this case $(\partial \pi)_{\Omega_h}$ and $(\partial \pi)_o$ are orthogonal, and the same holds for $\pi_h^k$.
		Then we have
		\begin{equation*}
		\begin{aligned}
			g^{-1} (g -g^k_h ) = &g^{-1} ((\partial \pi)^\top \partial \pi  - (\partial \pi_h^k)^\top \partial \pi_h^k )\\
		=&g^{-1} \frac{1}{2}\left(((\partial \pi)^\top  + (\partial \pi_h^k)^\top )(\partial \pi -\partial \pi_h^k ) + ((\partial \pi)^\top  - (\partial \pi_h^k)^\top )(\partial \pi + \partial \pi_h^k )\right) \\
		=&g^{-1} \frac{1}{2}\left(((\partial \pi)^\top  + (\partial \pi_h^k)^\top )_{\Omega_h}(\partial \pi -\partial \pi_h^k )_{\Omega_h} + ((\partial \pi)^\top  - (\partial \pi_h^k)^\top )_{\Omega_h}(\partial \pi + \partial \pi_h^k )_{\Omega_h} \right) \\
		&+g^{-1} \frac{1}{2}\left(((\partial \pi)^\top  + (\partial \pi_h^k)^\top )_{o}(\partial \pi -\partial \pi_h^k )_{o} + ((\partial \pi)^\top  - (\partial \pi_h^k)^\top )_{o}(\partial \pi + \partial \pi_h^k )_{o} \right) ,
		\end{aligned}  
		\end{equation*}
		where we use the decomposition in Definition \ref{def:decomposition}.
		Note that due to the assumptions on $\Gamma$, $g^{-1}$ is uniformly bounded independent of $h$. 
		Then using the estimate from Lemma \ref{lem:inter_error}, we have the following  
		\[ \norm{(\partial \pi -\partial \pi_h^k )_{\Omega_h}}_{L^\infty(\Omega_h)}\lesssim h^{k+1}  \text{ and }  \norm{((\partial \pi)^\top  - (\partial \pi_h^k)^\top )_{o}}_{L^\infty(\Omega_h)}\lesssim h^k\]
		and $  \norm{(\partial \pi +\partial \pi_h^k )_{\Omega_h}}_{L^\infty(\Omega_h)}\leq C$ is uniformly bounded for some constant $C$ independent of $h$, and
		\[  \norm{((\partial \pi)^\top  + (\partial \pi_h^k)^\top )_{o}}_{L^\infty(\Omega_h)}\lesssim h .\]
		These together give us the first estimate.
		The second and the third estimates are consequences derived immediately from the first estimate.
		We skip the details here.
		For the forth one, notice that $n=\partial \pi e_n$ and $n_h^k=\partial \pi_h^k e_{n_h}$ where $e_n$ and $e_{n_h}$ both are $2\times 1$ vectors defined on the edges of the triangles in the common planar parameter domain $\Omega_h$.
		Since both $\Gamma_h$ and $\hat{\Gamma}_h^k$ have uniformly bounded curvature, then the lengths of both $e_n$ and $e_{n_h}$ have uniform upper bound and also uniform lower bound larger than zero.
		Thus
		\[g^{-1} (\partial \pi )^\top n - (g_h^k)^{-1} (\partial \pi^k_h )^\top n^k_h = g^{-1}g e_n-  (g_h^k)^{-1} g_h^k e_{n_h} =e_n-e_{n_h}.\]
		On the other hand, since both $n$ and $n_h$ are unitary vectors, and $g$ is positive definite, we have
		\[ e_n^\top g e_n= e_{n_h}^\top g_h^k e_{n_h}=1 \quad \Rightarrow \quad (e_n-e_{n_h}) ^\top g (e_n+e_{n_h}) = e_{n_h}^\top (g_h^k -g) e_{n_h}. \]
		Multiply the Moore-Penrose pseudoinverse of $g (e_n+e_{n_h}) $, denoted by $\phi$, on both sides. For a vector, such $\phi$ exists and unique, and $\phi= [(e_n+e_{n_h})^\top g g (e_n+e_{n_h})]^{-1}(e_n+e_{n_h})^\top g $.   
		We get the equation
		\[(e_n-e_{n_h}) ^\top = e_{n_h}^\top (g_h^k -g) e_{n_h} \phi.
		\]
		Then due to the uniform boundedness of the relevant quantities $ e_{n_h}$ and $\phi$, we derive
		\[\norm{e_n-e_{n_h}}_{L^\infty} \lesssim \norm{g-g_h^k}_{L^\infty} , \]
		which then gives us the estimate that
		\[ \norm{g^{-1} (\partial \pi )^\top n - (g_h^k)^{-1} (\partial \pi^k_h )^\top n^k_h}_{L^\infty}= \norm{e_n-e_{n_h}}_{L^\infty(\Omega_h)}  =\mathcal{O}(h^{k+1}).\]
	\end{proof}

\begin{remark}\label{rem:discrepancy}
We note that surprisingly the two error estimates of the Jacobian $\norm{\partial \pi -\partial \pi_h^k }_{L^\infty(\Omega_h)} $ and the metric tensors $ \norm{g-g_h^k}_{L^\infty}$ are not of the same order, even though the metric tensors are represented by product of Jacobian and its transpose given in \eqref{eq:metric_tensor}. The extra order for the metric tensor approximation is due to the assumption that the local patches parameterized by $\Omega_h$ are in the $\mathcal{O}(h^2)$ neighborhood of $\Omega_h$, which results in the estimate \eqref{eq:geo_approx2}. For the interpolating polynomial patches of a smooth manifold as in the literature, the estimate in \eqref{eq:geo_approx2} is automatically fulfilled, which, however, seems to be not granted for arbitrary approximations. \emph{Notice that, with only the condition in \eqref{eq:geo_approx1}, it is in general not sufficient to prove the results in Theorem \ref{thm:geo_error}.} To show that \eqref{eq:geo_approx2} is indeed required, an immediate example one may think is a right-angled triangle with two orthogonal edges of length $h$, while another right-angled triangle in the same plane with two orthogonal edges of length $h+h^{k+1}$. For such a pair, the inequality of \eqref{eq:geo_approx1} is satisfied but not the one in \eqref{eq:geo_approx2}. Then the ratio of the areas of the two triangles is $1+2h^{k}+h^{2k}$, i.e., $\norm{\frac{\sqrt{\abs{g}}-\sqrt{\abs{g^k_h} }}{\sqrt{\abs{g}}}}_{L^\infty}  =2h^{k}+h^{2k}$. That is it has one order less than the result in Theorem \ref{thm:geo_error}.
\end{remark}

	\section{DG methods in curved domain and their error analysis}
	\label{sec:Laplace}
To illustrate the main idea better, we focus on two important exemplary elliptical PDEs on $\Gamma$. The first one is the Laplace-Beltrami equation: Given $f\in L^{2}(\Gamma)$, find $u\in H^2(\Gamma)$ such that
\begin{equation}\label{eq:Lap_Bel}
- \Delta_g u + u = f,
\end{equation}
where $g$ is the Riemannian metric of $\Gamma$. Note that here we introduce a reaction term to take care of the uniqueness of the solution. The weak solution of problem \eqref{eq:Lap_Bel} is to find $u\in H^1(\Gamma)$ such that 
\begin{equation}\label{eq:weak_Lap_Bel}
\mathcal{A}(u, v)= (f, v), \quad \text{ for all }  v\in H^1(\Gamma),
\end{equation}
where 
\begin{align*}
\mathcal{A}(u, v) = \int_\Gamma (\nabla_g u\cdot \nabla_g v + u v) dA_g,\quad \text{ and } \quad (f, v) = \int_\Gamma f v dA_g.
\end{align*}
Here $A_g$ denotes the surface measure on $\Gamma$ taking into account the Riemannian metric $g$. Later we use $E_g$  to denote the restricted measure on edges of the curved triangles. We will adopt similar notations like $A_{g_h^k}$ and  $E_{g_h^k}$ on the discrete surface and its edges. 

The second model equation is the eigenvalue problem associated with the Laplace-Beltrami operator on $\Gamma$: find pairs $(\lambda,u)\in (\mathbb{R}^+, H^2(\Gamma))$ such that
\begin{equation}\label{eq:eigen_Lap_Bel}
- \Delta_g u = \lambda u.
\end{equation}
The eigenvalues are assumed to be ordered so that $0= \lambda_1\leq \lambda_2\leq \cdots \leq \lambda_n\leq \cdots$, and the corresponding eigenfunctions are normalized to satisfy $\norm{u_i}_{L^2(\Gamma)}=1$.
The weak formulation of \eqref{eq:eigen_Lap_Bel} reads as: find $(\lambda,u)\in (\mathbb{R}^+, H^1(\Gamma))$ with $\norm{u}_{L^2(\Gamma)}=1$ such that
\begin{equation}\label{eq:weak_eigen_Lap_Bel}
\int_\Gamma \nabla_g u \cdot \nabla_g v dA_g= \lambda \int_\Gamma  u v dA_g  \quad \text{ for all }  v\in H^1(\Gamma).
\end{equation}
Well-posedness of the above problems has been discussed, for instance, in \cite{Aubin1982}. Our goal here is to analyze DG methods when $\Gamma$ is approximated by $\hat{\Gamma}_h^k$. We first introduce relevant function spaces for DG methods.

\subsection{Function spaces associated to DG methods}
In order to conduct an error analysis for the DG method formulated in \eqref{eq:DG_bilinear_1}, we consider the following function spaces and associated norms. They both involve the geometry by $\hat{\Gamma}_h^k$ and $\Gamma_h$, the patchwise manifold and the triangulated original manifold, respectively. We remind that $\Gamma_h$ is just a partition of $\Gamma$ by curved triangles of diameter $h$. We use the notation $H_b^r$ to denote the broken Sobolev spaces
\begin{equation}\label{eq:piecewise_Hk}
    H_b^r(\hat{\Gamma}_h^k):=\set{v:\hat{\Gamma}_h^k \to \R, \text{ and } v|_{\hat{\Gamma}_h^{k,j}}\in H^r(\hat{\Gamma}_h^{k,j}), \text{ for all } j \in J} ,
\end{equation}
where $H^r$ denotes the standard Sobolev space, and $r$ is the order of the differentiation.
The definition in \eqref{eq:piecewise_Hk} can similarly apply to $\Gamma$ by considering each of the corresponding curved triangles $\Gamma_h^j$.
The DG finite element function space is given in an isoparametric way through the parameter domain $\Omega_h$. Precisely we consider piecewise polynomial functions on every parametric triangle patch $\Omega_h^j$, which are then pulled back to the corresponding surface patches:
\begin{equation}
    \hat{S}^{k,l}_h:=\set{v\in L^2(\hat{\Gamma}_h^k): v|_{\hat{\Gamma}_h^{k,j}} = \bar{v}\circ (\pi_h^k)^{-1} \text{ where } \bar{v} \in P^l(\Omega_h^j) \text{ for all } j\in J},
\end{equation}
where $l$ is the polynomial order for function values, and $k$ is the polynomial order for the local patches in $\hat{\Gamma}_h^{k,j}$.
The counterpart on the exact manifold $\Gamma_h$ is denoted by 
\begin{equation}
    S^l_h:=\set{v\in L^2(\Gamma_h): v|_{\Gamma_h^j} = \bar{v}\circ (\pi)^{-1} \text{ where } \bar{v} \in P^l(\Omega_h^j) \text{ for all } j\in J}.
\end{equation}
It is not hard to see that $\hat{S}^{k,l}_h\subset H_b^l(\hat{\Gamma}_h^k)$ and $S^l_h\subset H_b^l(\Gamma_h)$, respectively.
We introduce the following spaces 
\[ V_h^{l}:=H^l(\Gamma_h)+ H_b^l(\Gamma_h) \quad \text{ and }  \quad V_h^{k,l}:= H^l(\hat{\Gamma}_h^k)+ H_b^l(\hat{\Gamma}_h^k).\]
Their norms can characterize the DG spaces, which are often called DG norms. Here we write down the one for the former
\[ \norm{v_h}_{V_h^{l}}^2=\sum_j  \norm{v_h}^2_{H^1(\Gamma^j_h)} + \sum_{e_h\in \mathcal{E}_h} \beta_h  \norm{[v_h]}^2_{L^2(e_h)} \quad \text{ for every } v_h \in V_h^{l}. \]
The one for the latter is similar and thus omitted. For simplicity, we denote
\begin{equation}\label{eq:trace_norm}
    \norm{v_h}_*^2=\sum_{e_h\in \mathcal{E}_h} \beta_h  \norm{[v_h]}_{L^2(e_h)}^2.
\end{equation}
To compare functions defined on $\Gamma_h$ and $\hat{\Gamma}_h^k$, we rely on the pullback and pushforward operations, which are precisely defined as follows:
\begin{equation}\label{eq:pullback}
    T^k_h u_h:=u_h\circ \pi_h^k\circ (\pi)^{-1} .
\end{equation}
Note that the pullback $T^k_h$ is a bijection and both itself and its inverse the pushforward $(T^k_h)^{-1}$ are bounded operators in the sense that:
	\begin{lemma}
		Let $T^k_h$ and its inverse $(T^k_h)^{-1}$ be defined as \eqref{eq:pullback}, then
		\begin{equation}\label{eq:pullback_equiv}
			\norm{T^k_h u_h}_{H^m(\Gamma_h)} \simeq \norm{u_h}_{H^m(\hat{\Gamma}^k_h)}  \quad \text{ for all } 0\leq m \leq l+1.
		\end{equation}
		By $a \simeq b$ we mean that there exist positive constants $C_1$ and $C_2$ such that $a\leq C_1 b$ and $b\leq C_2 a$.
	\end{lemma}
	
	%\[ \norm{u_h}_{DG(\hat{\Gamma}^k_h)}^2=\sum_j \left( \norm{u_h}^2_{H^1(\hat{\Gamma}^{k,j}_h)} + \frac{\beta_h}{2}\left( \norm{[u_h]}^2_{L^2((e^{k,j}_h)^+)}+ \norm{[u_h]}^2_{L^2((e^{k,j}_h)^-)} \right)  \right) \text{ for every } u_h \in \hat{S}_{h,k}. \]
	
	\subsection{DG methods for PDEs on curved domain}
	%Discontinuous Galerkin methods for elliptical PDEs can be traced back to several seminal works, e.g., \cite{Arn82, DiEr2012, Rivi2008,HeWa2008}.
 
	We shall focus on an interior penalty discontinuous Galerkin method (IPDG) for  PDEs on patchwise manifolds.
	The standard formulations of DG methods in a surface setting have been studied in \cite{DedMadSti13,AntDedMadStaStiVer15}, a bilinear form of which is recalled as follows: For $u_h,v_h\in V_h^l$,
	\begin{equation}\label{eq:DG_ambient}
	\begin{aligned}
	\mathcal{A}_h^{k,l}(u_h,v_h):=&\sum_{j} \int_{\hat{\Gamma}^{k,j}_h} \left(\nabla_{g^k_h} u_h \cdot \nabla_{g^k_h} v_h  + u_h v_h  \right)dA_{g^k_h}\\
	- \sum_{\hat{e}^{k}_h \in \hat{ \mathcal{E}}_h} &\int_{ \hat{e}^{k}_h}  \left( \{(\nabla_{g^k_h} u_h\cdot n_h)\} [v_h] + \{(\nabla_{g^k_h} v_h\cdot n_h)\}   [u_h]   -  \frac{\beta}{h} [u_h][v_h]  \right)dE_{g^k_h}.
	\end{aligned}
	\end{equation}
	Here $\hat{ \mathcal{E}}_h$ denotes the edge set containing all the edges from the curved triangular faces in $\hat{\Gamma}_h^k$, $\beta$ denotes the  stabilization parameter, and $n_h$ is the outer unit normal vector orthogonal to the edges $\hat{e}^{k}_h$ and tangential to $\hat{\Gamma}_h^k$.
	For quantity $q$ defined on $(\hat{e}_h^{k,j})^\pm \in \partial \hat{\Gamma}_h^{k,j}$, the averaging is defined to be 
	$\set{q}:=\frac{1}{2}(q^+ +q^-)$ and the jump  is defined to be $[q]:=q^+ -q^-$. 
	When the two edges do not coincide, the averaging and jump functions have to be understood by pulling back them onto the parameter domain, i.e., $\set{\bar{q}}=\frac{1}{2}(\bar{q}^+ +\bar{q}^-)$, and $[\bar{q}]:=\bar{q}^+ - \bar{q}^-$. Here $\bar{q} = q\circ \pi_h^k$.
	%Note that the bilinear form \eqref{eq:DG_ambient} has been considered already in \cite{AntDedMadStaStiVer15} though different notations are used there.
	The main difference  is that  in \cite{AntDedMadStaStiVer15}, the exact manifolds are assumed to be known and $\hat{\Gamma}_h^k$ is assumed to be piecewise polynomial interpolation of $\Gamma_h$.
	However, different situations are confronted in general for patchwise manifolds.
	%First of all, $\Gamma$ is not known, and we need to  construct such an approximation from a set of scattered points in order to apply a Galerkin type numerical method. This shall be  discussed  in Section \ref{subsec:geo_algo}.
	It might be the case that the patches in $\hat{\Gamma}_h^k$ share no common edges with their neighbors. This is usually the case resulting from the local geometric reconstruction of the manifolds through, e.g., point clouds. Therefore the DG methods in the literature can not be applied directly. In this regard, we need to write down a more general computational formula.
	To do so, we pull back all the calculations onto the parametric domain $\Omega_h$.
	Then we have the following bilinear form for general problems on manifolds: For $u_h, v_h \in V_h^{k,l}$,
	\begin{equation}\label{eq:DG_bilinear_1}
	\begin{aligned}
	& \mathcal{A}_h^{k,l}(u_h,v_h):= \\
	&\sum_j \int_{\Omega_h^j}\left(( \nabla \bar{u}_h )^\top (g_h^k)^{-1}\nabla \bar{v}_h +  \bar{u}_h \bar{v}_h \right) \sqrt{\abs{g_h^k}} dA+ \frac{\beta}{h}  \sum_{e_h\in \bar{\mathcal{E}}_h} \int_{e_h} [\bar{u}_h][\bar{v}_h ] \{ l_{g_h^k}\}dE\\
	& -\sum_{e_h\in \bar{\mathcal{E}}_h} \int_{e_h} \{ ( \nabla \bar{u}_h )^\top (g_h^k)^{-1}(\partial \pi_h^k)^\top n_h l_{g_h^k} \} [\bar{v}_h]+\{ ( \nabla \bar{v}_h )^\top (g_h^k)^{-1}(\partial \pi_h^k)^\top n_h l_{g_h^k}\} [\bar{u}_h ] dE .
	\end{aligned}
	\end{equation}
 Here and in the following, we use $\bar{w}:\Omega_h\to \R$ to denote the pullback of $w$ defined on $\Gamma_h$ or $\hat{\Gamma}_h^k$.
	Note that in the parametric domain, $e_h$ between two neighbored triangles are unique, which is similar to the standard DG methods. However, different metrics apply to the common edge. This has been reflected in the bilinear form \eqref{eq:DG_bilinear_1} where $l_{g_h^k}^+$ and $l_{g_h^k}^-$ might be different.
	
	In order to have a comparison with the DG bilinear form on the exact manifolds, we also consider the following  parametric presentation on $\Gamma_h$:  For $u_h,v_h\in V_h^l$,
	\begin{equation}\label{eq:DG_bilinear_1_exact}
	\begin{aligned}
	\mathcal{A}_h^{l}(u_h,v_h):= &\sum_j \int_{\Omega_h^j}\left(( \nabla \bar{u}_h )^\top g^{-1}\nabla \bar{v}_h + \bar{u}_h \bar{v}_h \right)\sqrt{\abs{g}} dA+ \frac{\beta}{h}  \sum_{e_h\in \bar{\mathcal{E}}_h} \int_{e_h} [\bar{u}_h][\bar{v}_h ]  l_{g} dE\\
	-& \sum_{e_h\in \bar{\mathcal{E}}_h} \int_{e_h} \left(\{ ( \nabla \bar{u}_h)^\top g^{-1}(\partial \pi)^\top n\}[\bar{v}_h]+ \{( \nabla \bar{v}_h )^\top g^{-1}(\partial \pi)^\top n\} [\bar{u}_h ] \right) l_{g} dE .
	\end{aligned}
	\end{equation}
	In the following, we denote $\beta_h:=\frac{\beta}{h}$.
	Note that the difference between \eqref{eq:DG_bilinear_1_exact} and \eqref{eq:DG_bilinear_1} is that the edge metric scaling $l_g$ is the same on every shared edge while $l_{g_h^k}$ in general might be not coincide. In that case the formula in \eqref{eq:DG_bilinear_1_exact} is just a rewriting of the formula in \eqref{eq:DG_ambient} by replacing the metric tensor $g$ with $g_h^k$.

	%\section{Error estimates of DG methods}
	%\label{sec:Laplace}

	%Since the IPDG method in a surface setting has been studied in \cite{AntDedMadStaStiVer15}, the continuity and the coercivity of the bilinear form is rather the same. We simply refer to the results there.
	\subsection{Error analysis for Laplace-Beltrami equation}
	To develop the convergence analysis, we follow the general strategy in \cite{AntDedMadStaStiVer15,DedMadSti13}, but mostly emphasize the geometric error analysis part.
	It shows that the geometric error are exactly caused by the approximation of the metric tensor.
	We shall refer to a few lemmas from \cite{AntDedMadStaStiVer15} which are proposed for DG methods on continuous surfaces.
	\begin{lemma}[DG trace inequality]\label{lem:DG_trace}
		For every $e^j_h\in \partial \Gamma_h^j$, and for sufficiently small $h$, we have 
		\begin{equation}\label{eq:trace_ineq}
		\norm{v_h}_{L^2(e^j_h)}^2\lesssim h^{-1}\norm{v_h}_{L^2(\Gamma_h^j)}^2 \; \text{ and } \; \norm{\nabla_g v_h}_{L^2(e^j_h)}^2\lesssim h^{-1} \norm{\nabla_g v_h}_{L^2(\Gamma_h^j)}^2.
		\end{equation}
	\end{lemma}
	\begin{lemma}[Boundedness and coercivity of DG bilinear form]\label{lem:DG_coerciv}
		The DG bilinear form $\mathcal{A}_h^l$ defined on $\Gamma_h$ is bounded and coercive with respect to DG norm provided that the  stabilization parameter $\beta$ is large enough, i.e., for all $u_h, v_h \in V_h^l$
		\begin{equation}
		\mathcal{A}_h^l(u_h,v_h)\lesssim \norm{u_h}_{V_h^l}\norm{v_h}_{V_h^l}\quad \text{ and }\quad \mathcal{A}_h^l(u_h,u_h)\gtrsim \norm{u_h}^2_{V_h^l}.
		\end{equation}
	\end{lemma}
	We emphasis that $\mathcal{A}_h^l(\cdot,\cdot)$ is the DG formulation of continuous surfaces, which is used for the error analysis in the following.
	
	With the above preparation, we are now ready to present our main result on the error bounds of the IPDG method.  
	\begin{theorem}\label{thm:convergence}
		Let $\Gamma_h$ be a partition of a smooth manifold, and $\hat{\Gamma}_h^k$ be a $k$-th order patchwise manifold which approximate $\Gamma_{h}$. Let $u^{k,l}_h\in \hat{S}^{k,l}_h$ be the solution by DG discretization of \eqref{eq:weak_Lap_Bel}, and  $u\in H^{l+1}(\Gamma_h) $ be the solution of \eqref{eq:Lap_Bel}.
		Then we have the following convergence result:
		\begin{equation}\label{eq:convergence}
		\norm{\hat{u}^{k,l}_h-u}_{L^2(\Gamma_h)} + h \norm{\hat{u}^{k,l}_h -u}_{V_h^{l}}\lesssim h^{\min\set{k,l}+1}(\norm{f}_{L^2(\Gamma_h)}+ \norm{u}_{H^{l+1}(\Gamma_h)}),
		\end{equation}
		where $\hat{u}^{k,l}_h:=T^k_h u^{k,l}_h$  is the pullback of the function $u^{k,l}_h$ from $\hat{\Gamma}_h^k$ to $\Gamma_h$.
	\end{theorem}
	\begin{proof}
		We sketch some of the major steps for a proof strategy, while highlight the part which is different to the one in \cite{AntDedMadStaStiVer15}. We first show the estimate on the DG norm of the error $ \norm{\hat{u}^{k,l}_h -u}_{V_h^{l}}$. After that we proceed with the $L^2(\Gamma_h)$ norm error estimate. 
		
		\textbf{Step 1:} Using triangle inequality, this can be decomposed into:
		\begin{equation}\label{eq:DG_norm_term}
		\norm{\hat{u}^{k,l}_h-u}_{V_h^{l}} \leq \norm{I_h^l u-u}_{V_h^{l}} + \norm{I_h^l u - \hat{u}^{k,l}_h}_{V_h^{l}},
		\end{equation}
		where $I_h^l$ is the operator interpolating the function value of $u\circ \pi$ using $l$th order polynomial on parameter domain $\Omega_{h}$ and then pull back to $\Gamma_h$.
		
		\textbf{Step 2:} The first term on the right hand side of \eqref{eq:DG_norm_term} can be estimated using interpolation error estimate and the trace inequality\eqref{eq:trace_ineq}. This gives 
		\[  \norm{I_h^l u-u}_{V_h^{l}} \lesssim h^{l} \norm{u}_{H^{l+1}(\Gamma_h)}. \]
		Be careful here we have interpolation both on the triangular faces and also on the edges.
		
		\textbf{Step 3:} The second term on the right hand side of \eqref{eq:DG_norm_term} needs a bit more work.
		Using the stability estimate in Lemma \ref{lem:DG_coerciv}
		\[\norm{I_h^l u - \hat{u}^{k,l}_h}^2_{V_h^{l}} \lesssim \mathcal{A}_h^{l}(I_h^l u - \hat{u}^{k,l}_h, I_h^l u - \hat{u}^{k,l}_h) ,\]
		where $\mathcal{A}_h^l(\cdot,\cdot)$ is the bilinear forms induced by the DG discretization of the weak formulation of the PDEs, e.g., Laplace-Beltrami problem. Note that this is formulated on the underlying exact manifold $\Gamma_h$.
		
		\textbf{Step 4:} The key point of the proof is to estimate $\mathcal{A}_h^{l}(I_h^l u - \hat{u}^{k,l}_h, I_h^l u - \hat{u}^{k,l}_h) $.  Recall
		\begin{equation}\label{eq:bilinear_estimate}
		\mathcal{A}_h^{l}(I_h^l u - \hat{u}^{k,l}_h, I_h^l u - \hat{u}^{k,l}_h)= \mathcal{A}_h^{l}(I_h^l u - u, I_h^l u - \hat{u}^{k,l}_h)+ \mathcal{A}_h^{l}(u - \hat{u}^{k,l}_h, I_h^l u - \hat{u}^{k,l}_h).
		\end{equation}
		The first term of the right hand side can be bounded by interpolation estimate and Lemma \ref{lem:DG_coerciv}:
		\begin{equation}\label{eq:bilinear_interpolate}
		\mathcal{A}_h^{l}(I_h^l u - u, v_h) \lesssim h^{l}\norm{u}_{H^{l+1}(\Gamma_h)}\norm{v_h}_{V_h^{l}},
		\end{equation}
		for $v_h\in V_h^l$.
		The second term on the right-hand side of \eqref{eq:bilinear_estimate} looks like a Galerkin orthogonality residual in conformal FEM.
		In order to analyze it, we consider the the weak formulation and the DG approximation of the Laplace-Beltrami equation, which consists of
		\[ \mathcal{A}_h^{l}(u, v_h)=\sum_j \int_{\Gamma_h^j} f  v_h dA_g \]
		and 
		\[ \mathcal{A}_h^{k,l}(u^{k,l}_h,  (T_h^k)^{-1} v_h)=\sum_j \int_{\hat{\Gamma}^{k,j}_h} (T_h^k)^{-1} f  (T_h^k)^{-1} v_h dA_{g^k_h},\]
		where $\mathcal{A}_h^{l}(\cdot,\cdot)$ and $\mathcal{A}_h^{k,l}(\cdot,\cdot)$ denotes the DG bilinear forms on $\Gamma_h$ and $\hat{\Gamma}^k_h$, respectively. Their precise forms have been given in \eqref{eq:DG_bilinear_1_exact} and \eqref{eq:DG_bilinear_1}.
		
		Note that 
		\[\sum_j \int_{\hat{\Gamma}^{k,j}_h} (T_h^k)^{-1} f  (T_h^k)^{-1} v_h dA_{g^k_h}=\sum_j \int_{\Gamma_h^j}  f   v_h \sigma_A dA_g,\]
		where $\sigma_A =\frac{dA_{g^k_h}}{dA_g}=\frac{\sqrt{\abs{g^k_h}}}{\sqrt{\abs{g}}}$. 
		This shows
		\[   \mathcal{A}_h^{l}(u, v_h)- \mathcal{A}_h^{k,l}(u^{k,l}_h,  (T_h^k)^{-1} v_h)= \sum_j \int_{\Gamma_h^j}  f   v_h (1-\sigma_A) dA_g.   \]
		Now we consider to calculate the quantities on the parametric domain $\Omega_h$
		\[ 
		\begin{aligned}
		&\mathcal{A}_h^{l} (T_h^k u^{k,l}_h,   v_h)\\
		%&\sum_j \int_{\Gamma^j_h} \nabla_g  T_h^k u_h \cdot \nabla_g v_h dA_g -\int_{\partial \Gamma_h^j}\{(\nabla_g T_h^k u_h\cdot n)\} \{v_h\}  dE_g + \beta_h \int_{\partial \Gamma_h^j} [T_h^k  u_h][v_h]  dE_g ,\\
		=&\sum_j \int_{\Omega_h^j} \left(  (\nabla \bar{u}^{k,l}_h )^\top g^{-1}\nabla \bar{v}_h+ \bar{u}^{k,l}_h \bar{v}_h\right) \sqrt{\abs{g}} dA+ \beta_h  \sum_{\bar{e}_h\in \bar{\mathcal{E}}_h} \int_{\bar{e}_h} [\bar{u}^{k,l}_h][\bar{v}_h ]  l_{g} dE\\
		&-\sum_{\bar{e}_h\in \bar{\mathcal{E}}_h} \int_{\bar{e}_h} \left(\{ ( \nabla \bar{u}^{k,l}_h)^\top g^{-1}\partial \pi^\top n\}[\bar{v}_h]+ \{( \nabla \bar{v}_h )^\top g^{-1}\partial \pi^\top n\} [\bar{u}^{k,l}_h ] \right) l_{g} dE 
		\end{aligned}
		\]
		while 
		\[ 
		\begin{aligned}
		&\mathcal{A}_h^{k,l} ( u^{k,l}_h, (T_h^k)^{-1} v_h)\\
		%&\sum_j \int_{\hat{\Gamma}^{k,j}_h} \nabla_{g^k_h} u_h \cdot \nabla_{g^k_h} (T_h^k)^{-1} v_h dA_{g^k_h} -\int_{\partial \hat{\Gamma}^{k,j}_h}\{(\nabla_{g^k_h} u_h\cdot n_h)\} \{(T_h^k)^{-1}  v_h\}  dE_{g^k_h}\\
		%& \quad+ \beta_h \int_{\partial \hat{\Gamma}^{k,j}_h} [u_h][T_h^k)^{-1}v_h]  dE_{g^k_h},\\
		&=\sum_j  \int_{\Omega_h^j} \left((\nabla \bar{u}^{k,l}_h)^\top (g^k_h)^{-1} \nabla \bar{v}_h + \bar{u}^{k,l}_h \bar{v}_h\right) \sqrt{\abs{g^k_h}}dA  + \beta_h \sum_{\bar{e}_h\in \bar{\mathcal{E}}_h} \int_{ \bar{e}_h} [\bar{u}^{k,l}_h] [\bar{v}_h] \{l_{g^k_h} \}dE -  \\
		&\sum_{\bar{e}_h\in \bar{\mathcal{E}}_h} \int_{\bar{e}_h} \{(\nabla \bar{u}^{k,l}_h)^\top (g^k_h)^{-1}  (\partial \pi_h^k)^\top n_h l_{g^k_h}\}[ \bar{v}_h  ] + \{(\nabla \bar{v}_h)^\top (g^k_h)^{-1}  (\partial \pi_h^k)^\top n_h l_{g^k_h} \} [\bar{u}^{k,l}_h  ]dE .
		\end{aligned}
		\]
		We define 
		\begin{equation}\label{eq:DG_ortho1}
		\mathcal{R}(v_h):= \mathcal{A}_h^{l}(u - \hat{u}^{k,l}_h, v_h).
		\end{equation}
		which is similar but different to the Galerkin orthogonality term.  
		We have that
		\begin{equation*}
		\begin{aligned}
		\mathcal{R}(v_h)=&\mathcal{A}_h^{k,l}( u^{k,l}_h, (T_h^k)^{-1} v_h)-\mathcal{A}_h^{l} (T_h^k u^{k,l}_h,   v_h) 
		+ \mathcal{A}_h^{l}(u, v_h)- \mathcal{A}_h^{k,l}(u^{k,l}_h,  (T_h^k)^{-1} v_h)
		\\
		%=& \sum_j  \int_{\Omega_h^j} (\nabla \bar{u}- \nabla \bar{u}_h)^\top g^{-1} \nabla \bar{v}_h \sqrt{\abs{g}}dA - \int_{ \partial \Omega_h^j} \{(\nabla \bar{u}- \nabla \bar{u}_h)^\top g^{-1}  (\partial \pi)^\top n\}\{\bar{v}_h \} l_{g} dE \\
		%& \quad+ \beta_h \int_{ \partial \Omega_h^j} [\bar{u}-\bar{u}_h] [\bar{v}_h] l_{g} dE\\
		=& \mathcal{A}_h^{k,l}( u^{k,l}_h, (T_h^k)^{-1} v_h)-\mathcal{A}_h^{l} (T_h^k u^{k,l}_h,   v_h) + \sum_j \int_{\Gamma_h^j}  f   v_h (1-\sigma_A) dA_g,
		\end{aligned}
		\end{equation*}
		which leads to
		\begin{equation}\label{eq:R_error}
		\begin{aligned}
		&	\mathcal{R}(v_h)=\\
		&\sum_j  \int_{\Omega_h^j} (\nabla \bar{u}^{k,l}_h)^\top \left((g^k_h)^{-1}-g^{-1}\right) \nabla \bar{v}_h \sqrt{\abs{g^k_h}} + (\nabla \bar{u}^{k,l}_h)^\top g^{-1}\nabla \bar{v}_h \left(\sqrt{\abs{g^k_h}} - \sqrt{\abs{g}}\right) dA   \\
		& -  \sum_{\bar{e}_h\in \bar{\mathcal{E}}_h}\int_{ \bar{e}_h} \{(\nabla \bar{u}^{k,l}_h)^\top \left( (g^k_h)^{-1}  (\partial \pi_h^k)^\top n_h - g^{-1}(\partial \pi )^\top n\right) l_{g^k_h} \}[\bar{v}_h  ]dE \\
		&  -  \sum_{\bar{e}_h\in \bar{\mathcal{E}}_h}\int_{ \bar{e}_h} \{(\nabla \bar{v}_h)^\top \left( (g^k_h)^{-1}  (\partial \pi_h^k)^\top n_h - g^{-1}(\partial \pi )^\top n\right) l_{g^k_h}\}[\bar{u}^{k,l}_h  ]dE \\
		& - \sum_{\bar{e}_h\in \bar{\mathcal{E}}_h} \int_{ \bar{e}_h} \{(\nabla \bar{u}^{k,l}_h)^\top g^{-1}(\partial \pi )^\top n \left(l_{g^k_h}-l_g\right)\}[\bar{v}_h ]  dE \\
		&  - \sum_{\bar{e}_h\in \bar{\mathcal{E}}_h} \int_{ \bar{e}_h} \{(\nabla \bar{v}_h)^\top g^{-1}(\partial \pi )^\top n \left(l_{g^k_h}-l_g\right)\}[\bar{u}^{k,l}_h ]  dE \\
		& + \beta_h \sum_{\bar{e}_h\in \bar{\mathcal{E}}_h}  \int_{\bar{e}_h} [\bar{u}^{k,l}_h] [\bar{v}_h]\{\left( l_{g^k_h}-l_g\right)\} dE
		+  \int_{\Omega_h^j}  ( \bar{u}^{k,l}_h \bar{v}_h- \bar{f}   \bar{v}_h ) \left(\sqrt{\abs{g^k_h}} - \sqrt{\abs{g}}  \right) dA.
		\end{aligned}
		\end{equation}
		Taking into account the geometric error estimates in Theorem \ref{thm:geo_error},  each of the terms on the right-hand side of \eqref{eq:R_error} can be estimated. The details of these estimates are given in Appendix \ref{app:DG_bilinear}.
		To summarize all the estimates there, we derive that
		\begin{equation}\label{eq:DG_ortho}
		\begin{aligned}
		\mathcal{R}(v_h) %\mathcal{A}_h^{l}(u - \hat{u}^{k,l}_h, v_h)
		%& \leq C_1h^{k+1}\abs{\mathcal{A}_h^{l}(T_h^k u^{k,l}_h,v_h)} + C_2 h^{k+1} \norm{f}_{L^2(\Gamma_h)}\norm{v_h}_{L^2(\Gamma_h)}\\
		& \lesssim  h^{k+1} \norm{T_h^k u^{k,l}_h}_{V_h^{l}}\norm{v_h}_{V_h^{l}}+  h^{k+1} \norm{f}_{L^2(\Gamma_h)}\norm{v_h}_{L^2(\Gamma_h)},
		\end{aligned}
		\end{equation}
		where $ \norm{T_h^k u^{k,l}_h}_{V_h^{l}}$ is uniformly bounded indepdent of $h$.
		Then we go back to \eqref{eq:bilinear_estimate} with \eqref{eq:bilinear_interpolate} $m=l+1$ and \eqref{eq:DG_ortho} to have the statement.
		In order to estimate $\norm{\hat{u}^{k,l}_h-u}_{L^2(\Gamma_h}$, we leverage the following dual problem of \eqref{eq:Lap_Bel}:
		\begin{equation}\label{eq:dual_estimate}
		-\Delta_g z +z=\hat{u}^{k,l}_h-u \; \text{ for }  \;  z\in H^2(\Gamma_h)\quad\text{  and } \quad \norm{z}_{H^2(\Gamma_h)}\lesssim \norm{\hat{u}^{k,l}_h-u}_{L^2(\Gamma_h)}.
		\end{equation}
		Testing with $\hat{u}^{k,l}_h-u$ on both sides gives us that
		\[\norm{\hat{u}^{k,l}_h-u}_{L^2(\Gamma_h)}^2=\mathcal{A}_h^l(z,\hat{u}^{k,l}_h-u)= \mathcal{A}_h^l(\hat{u}^{k,l}_h-u,z-I_h^lz) + \mathcal{A}_h^l(\hat{u}^{k,l}_h-u,I_h^lz). \]
		Then, applying standard boundedness estimate in Lemma \ref{lem:DG_coerciv} to the first term on the right-hand side, we have
		\[\mathcal{A}_h^l(\hat{u}^{k,l}_h-u,z-I_h^lz) 
		\leq \norm{\hat{u}^{k,l}_h-u}_{V_h^{l}}\norm{z-I_h^lz}_{V_h^{l}} 
		\leq Ch^{k+1}\norm{u}_{H^{k+1}(\Gamma_h)}\norm{z}_{H^2(\Gamma_h)}.\]
		For the second term, we use the estimate from  \eqref{eq:DG_ortho}, 
		\begin{equation*}
		\begin{aligned}
		&\mathcal{A}_h^l(\hat{u}^{k,l}_h-u,I_h^lz)=\mathcal{R}(I_h^lz) \\
		\lesssim & h^{k+1}\norm{T_h^k u^{k,l}_h}_{V_h^{l}}\norm{I_h^lz}_{V_h^{l}}+ h^{k+1} \norm{f}_{L^2(\Gamma)}\norm{I_h^lz}_{L^2(\Gamma_h) }\\
		\lesssim &  \left(h^{k+1}\norm{T_h^k u^{k,l}_h}_{V_h^{l}}+ h^{k+1} \norm{f}_{L^2(\Gamma)}\right) \norm{z}_{H^2(\Gamma_h)}.
		\end{aligned}
		\end{equation*}
		Combining with \eqref{eq:dual_estimate}, we end up with the expected estimates in \eqref{eq:convergence} for the $L^2$ norm. This concludes the proof in combination with the DG norm estimates.
	\end{proof}
	
	\subsection{Error analysis for the eigenvalue problem}
	Let the eigenpair $(\lambda,u)$ be the solution of the eigenvalue  problem \eqref{eq:eigen_Lap_Bel}, or equivalently \eqref{eq:weak_eigen_Lap_Bel}. The DG discretization of eigenvalue problem \eqref{eq:eigen_Lap_Bel} can be formulated as finding pairs $(\lambda^{k,l}_{h},u^{k,l}_{h})\in (\mathbb{R}, V_h^{k,l})$ such that
	\begin{equation}\label{eq:DG_eigen}
	\mathcal{A}_h^{k,l}(u^{k,l}_{h},v_h)=(\lambda^{k,l}_{h}+1)\sum_j \int_{\Omega_h^j} \bar{u}^{k,l}_{h}\bar{v}_h \sqrt{\abs{g_h^{k}}}dA, \quad \text{ for all }\quad  v_h \in V_h^{k,l}.
	\end{equation}
	Similarly,  the discrete eigenvalues of discrete problem \eqref{eq:DG_eigen} can be ordered as $0=\lambda^{k,l}_{h,1} < \lambda^{k,l}_{h,2} \le \cdots \le \lambda^{k,l}_{h,N},$ and the corresponding eigenfunctions $u^{k,l}_{h,i}$ ($i=1,\cdots, N$) can be normalized as 
	$(u^{k,l}_{h,i}, u^{k,l}_{h,j})= \delta_{ij}$ 
	where $N$ is the dimension of the DG function space.

	Let $G: L^2(\Gamma_h) \rightarrow H^1(\Gamma_h)\subset L^2(\Gamma_h)$ be the solution operator of the Laplace-Beltrami problem \eqref{eq:weak_Lap_Bel} which is defined as 
	\begin{equation}
	\mathcal{A}(Gf, v) = (f, v), \quad  \forall v\in H^1(\Gamma)
	\end{equation}
	and $G_h^{k,l}: L^2(\hat{\Gamma}_h^{k}) \to V^{k,l}_h \subset L^2(\hat{\Gamma}_h^{k})$ be the discrete solution operator which is defined as 
	\begin{equation}
	\mathcal{A}_h^{k,l}(G_h^{k,l}f, v_h) = ((T_h^k)^{-1}f, v_h), \quad  \forall v_h\in  V^{k,l}_h.
	\end{equation}
	Both $G$ and $G_h^{k,l}$ are self-adjoint. 
	Let  $\mu$ (or  $\mu_{h}^{k,l}$) be the eigenvalue  of $G$ (or $G_h^{k,l}$). Then, we can relate the  eigenvalue of \eqref{eq:eigen_Lap_Bel} and the eigenvalue of the operator $G$ by $\mu=\frac{1}{1+\lambda}$. Similarly, we have $ \mu_{}^{k,l} =\frac{1}{1+\lambda^{k,l}_{h}}$.  In addition, we have 
	\begin{equation}\label{eq:eigenvalue}
	\lambda^{k,l}_{h}-\lambda =(1+\lambda^{k,l}_{h})(1+\lambda) (\mu_i- \mu_{h}^{k,l} ).
	\end{equation}
	The following lemma tells that the eigenvalues are invariant under geometric transformation operator $T_h^k$ and its inverse $(T_h^k)^{-1}$.
	\begin{lemma}\label{lem:invariant}
		If $G \hat{v}= \mu \hat{v}$, then $ (T_h^k)^{-1}G \hat{v}= \mu (T_h^k)^{-1}\hat{v} $. Similarly, if $G_h^k v_h=\mu_h^k v_h$, then $T_h^k G_h^k v_h=\mu_h^k T_h^k v_h$.
	\end{lemma}
	\begin{proof}
		This is  shown using the fact that the operators $(T_h^k)^{\pm}$ and the scalar multiplication are commute, i.e., $\mu_h^k T_h^k v_h =T_h^k \mu_h^k v_h $ and  $\mu (T_h^k)^{-1}\hat{v} =(T_h^k)^{-1} \mu  \hat{v}$.
	\end{proof}

	Using the above lemma, we have the following result.
	\begin{lemma}\label{lem:invariant2}
		Let $\mu_{h}^{k,l}$ be an eigenvalue of the operator $G_h^{k,l}$ on $\hat{\Gamma}_h^k$ with eigenfunction $u_{h}^{k,l}$.
		Define a solution operator
		\begin{equation}\label{eq:DG_sol_operator}
		\tilde{G}_h^{k,l}:=T_h^kG_h^{k,l} (T_h^k)^{-1}\quad \text{ on } \Gamma_h.
		\end{equation}
		Then $\mu_{h}^{k,l}$ is  an  eigenvalue of $\tilde{G}_h^{k,l}$ and the corresponding eigenfunction is  $T_h^k u_{h}^{k,l} $.
	\end{lemma}
	\begin{proof}
		Let $\hat{u}_{h}^{k,l}:=  T_h^k u_{h}^{k,l}$.
		Notice that $ G_h^{k,l} u_{h}^{k,l}=\mu_{h}^{k,l} u_{h}^{k,l}$ which implies  that $G_h^{k,l} (T_h^k)^{-1} \hat{u}_{h}^{k,l} =\mu_h^k (T_h^k)^{-1} \hat{u}_{h}^{k,l}= (T_h^k)^{-1} \mu_h^k  \hat{u}_{h}^{k,l}$. It  tells that $T_h^k G_h^{k,l} (T_h^k)^{-1} \hat{u}_{h}^{k,l} = \mu_h^k  \hat{u}_{h}^{k,l}$, which completes the proof.
	\end{proof}

	To simplify the notation, we define
	\begin{equation}\label{eq:error_ope}
	\mathcal{E}_h^{k,l}(G):=G  - \tilde{G}_h^{k,l}.
	\end{equation}
	Using Theorem \ref{thm:convergence} for the source problem, we can show the operator approximation result
	\begin{lemma}\label{lem:solution_operator}
		Let $ \tilde{G}_h^{k,l}:L^2(\Gamma_h)\to V^l_h\subset L^2(\Gamma_h)$ be defined as in \eqref{eq:DG_sol_operator}, and $\mathcal{E}_h^{k,l}(G)$ be given as \eqref{eq:error_ope}.
		Then the following approximation properties hold:
		\begin{equation}\label{eq:operator_appr}
		\norm{\mathcal{E}_h^{k,l}(G)}_{\mathcal{L}(L^2(\Gamma_h))}=\mathcal{O}(h^{\min\set{k,l}+1}) .
		\end{equation}
	\end{lemma}
	In particular, we have 
	\begin{equation}\label{eq:operator_error} 
	\lim_{h\to 0 }\norm{\mathcal{E}_h^{k,l}(G)}_{\mathcal{L}(L^2(\Gamma_h))}  \to 0.
	\end{equation}

	Let $\rho(G)$ (or $\rho(\tilde{G}_h^{k,l})$)  denote the resolvent set of operator $G$ (or $\tilde{G}_h^{k,l}$), and $\sigma(G)$ (or $\sigma(\tilde{G}_h^{k,l})$) denote the spectrum set of operator $G$ (or $\tilde{G}_h^{k,l}$).  
	We define the spectral projection operator for an eigenvalue $\mu$: Let $\gamma\subset \rho(G)$ is a curve in the complex domain which encloses $\mu\in \sigma(G)$ but no other eigenvalues of $G$, the projection is
	\[E(\mu):=\frac{1}{2\pi \mathrm{i}} \int_{\gamma}(z-G)^{-1}dz .\]
	Let $R(E)\subset H^1(\Gamma_h)$ denote the range of $E$. 
	Let $\mu$ be an eigenvalue of the operator $G$, and it has algebraic multiplicity $m$. Then, combining  Theorem 9.1 in \cite{Bof10}  and \eqref{eq:operator_error} implies that no pollution of the spectrum, i.e. $\gamma$ encloses exactly 
	$m$ discrete  eigenvalues $\set{\mu_{h,i}^{k,l}}_{i=1}^m$ of $\tilde{G}_h^{k,l}$  approximating $\mu$.

	Using the Babu\v{s}ka-Osborn spectral approximation theory\cite{BabOsb91}, we have the following convergence results on eigenvalues and eigenfunctions.
	\begin{theorem}\label{thm:eigen_result}
		Let  $\mu^{k,l}_h$ be an eigenvalue converging to $\mu$. Let $\hat{u}^{k,l}_h$ be a unit eigenfunction of  $\tilde{G}_h^{k,l}$  associated with the eigenvalue $\mu^{k,l}_h$. Then there exists a unit eigenvector $u\in R(E)$  such that the following estimates hold
		\begin{align}
		&\|u - \hat{u}^{k,l}_h \|_{L^2(\Gamma_h)} \lesssim h^{\min\set{k,l}+1},\label{eq:eigenvalue_err} \\
		&\abs{\mu - \mu_h^{k,l}}\lesssim (h^{\min\set{k+1,2l}}),\label{eq:eigenfunction_err}\\
		&\abs{\lambda - \lambda_h^{k,l}}\lesssim (h^{\min\set{k+1,2l}}),\label{eq:neweigen_err}
		\end{align}
		where $l\geq 1$ and $k\geq 1$ are orders of the local polynomials for function and geometry approximations, respectively.
	\end{theorem}

	\begin{proof}
		For the eigenfunction approximation, Theorem 7.4 in \cite{BabOsb91} and  the approximation property \eqref{eq:operator_appr} already imply that
		\begin{equation}
		\|u - \hat{u}^{k,l}_h \|_{L^2(\Gamma_h)} \le 
		\| \mathcal{E}_h^{k,l}(G)|_{R(E)} \|_{L^2(\Gamma_h)}
		\le Ch^{\min\set{k,l}+1}.
		\end{equation}
		To establish the eigenvalue approximation result \eqref{eq:eigenfunction_err}, we apply 
		Theorem ~7.3 in \cite{BabOsb91}.   Let $v_1$, \ldots, $v_m$ be any orthonormal basis for $R(E)$.  Then, Theorem ~7.3 in \cite{BabOsb91} implies that there exists a constant $C$ such that
		\begin{equation}\label{eq:eigapp}
		|\mu - \mu_h| \lesssim  \sum_{j,k = 1}^m |(\mathcal{E}_h^{k,l}(G)v_j, v_k)| + \|\mathcal{E}_h^{k,l}(G)|_{R(E)}\|_{L^2(\Gamma_h)}^2.
		\end{equation}
		The second term can be bounded by $\mathcal{O}(h^{2\min\set{k,l}+2}) $.  For the first term, we have
		\begin{equation}\label{eq:error_eigen}
		\begin{aligned}
		&\left(\mathcal{E}_h^{k,l}(G)v_j,v_k\right)\\
		=& \mathcal{A} ( G v_j, Gv_k ) -  
		\mathcal{A}^l_h (\tilde{G}_h^{k,l} v_j, \tilde{G}_h^{k,l} v_k)
		+\mathcal{A}^l_h (\tilde{G}_h^{k,l} v_j, \tilde{G}_h^{k,l} v_k) - \left(\tilde{G}_h^{k,l} v_j,v_k\right).
		\end{aligned}
		\end{equation}
		Now we estimate each of the two paired terms on the right hand side of \eqref{eq:error_eigen} individually.  The first one gives
		\begin{equation}\label{eq:error_eigen1}
		\begin{aligned}
		&\mathcal{A} ( G v_j, Gv_k ) -  \mathcal{A}^l_h (\tilde{G}_h^{k,l} v_j, \tilde{G}_h^{k,l} v_k)
		=\mathcal{A}_h^l ( G v_j, Gv_k ) -  \mathcal{A}^l_h (\tilde{G}_h^{k,l} v_j, \tilde{G}_h^{k,l} v_k)  \\
		=& \mathcal{A}^l_h(G v_j- \tilde{G}_h^{k,l} v_j,  G v_k) + \mathcal{A}^l_h ( \tilde{G}_h^{k,l} v_j, G v_k -\tilde{G}_h^{k,l} v_k )  \\
		= & \mathcal{A}^l_h(\mathcal{E}_h^{k,l}(G)v_j,\tilde{G}_h^{k,l} v_k)+ \mathcal{A}^l_h (\tilde{G}_h^{k,l} v_j, \mathcal{E}_h^{k,l}(G)v_k )+ \mathcal{A}^l_h(\mathcal{E}_h^{k,l}(G)v_j, \mathcal{E}_h^{k,l}(G)v_k),
		\end{aligned}
		\end{equation}
		where the first equality holds since $Gu_j,\; Gv_k \in H^2(\Gamma_h)$ are continuous for two dimensional $\Gamma_h $.
		Notice that the third term on the last line of  \eqref{eq:error_eigen1} can be easily estimated by
		\[\mathcal{A}^l_h(\mathcal{E}_h^{k,l}(G)v_j, \mathcal{E}_h^{k,l}(G) v_k)\leq \norm{\mathcal{E}_h^{k,l}(G)v_j}_{V^l_h}\norm{ \mathcal{E}_h^{k,l}(G) v_k}_{V^l_h}.\]
		While for the first and the second terms on last line of \eqref{eq:error_eigen1}, we recall to  \eqref{eq:DG_ortho1}, which then indicates the following
		\[ \mathcal{A}^l_h(\mathcal{E}_h^{k,l}(G)u,\tilde{G}_h^{k,l} v)+ \mathcal{A}^l_h (\tilde{G}_h^{k,l} u, \mathcal{E}_h^{k,l}(G)v ) \leq \mathcal{R}(\tilde{G}_h^{k,l} v) +\mathcal{R}(\tilde{G}_h^{k,l} u) \leq C h^{k+1},\]
		where $\mathcal{R}(\cdot)$ is defined in \eqref{eq:DG_ortho1} and estimated in \eqref{eq:DG_ortho}.
		
		For the second term on the right hand side of \eqref{eq:error_eigen}, we have
		\begin{eqnarray*}
			& & \abs{\mathcal{A}^l_h (\tilde{G}_h^{k,l} u, \tilde{G}_h^{k,l} v) - \left(\tilde{G}_h^{k,l} u,v\right)}\\
			&\leq  & \abs{\mathcal{A}^l_h (\tilde{G}_h^{k,l} u, \tilde{G}_h^{k,l} v)-  \mathcal{A}^{k,l}_h (G_h^{k,l} T_h^k u, G_h^{k,l} T_h^k v)  +  \mathcal{A}^{k,l}_h (G_h^{k,l} T_h^k u, G_h^{k,l} T_h^k v) - \left( G_h^{k,l} T_h^ku,T_h^k v\right)}\\ 
			& &+\abs{\left( G_h^{k,l} T_h^ku,T_h^k v\right) - \left(\tilde{G}_h^{k,l} u,v\right)}\\
			&\lesssim &h^{k+1}+ \abs{ \underbrace{ \mathcal{A}^{k,l}_h (G_h^{k,l} T_h^k u, G_h^{k,l} T_h^k v) - \left( G_h^{k,l} T_h^ku,T_h^k v\right)}_{=0}}+h^{k+1},
		\end{eqnarray*}
		where the first and the third terms are due to the geometric error estimates from Theorem \ref{thm:geo_error}.
		Summarizing all the above terms, it concludes the proof  of \eqref{eq:eigenvalue_err}.  \eqref{eq:neweigen_err} is a direct consequence  of \eqref{eq:eigenfunction_err}   and \eqref{eq:eigenvalue}. 
	\end{proof}

\begin{remark}\label{rem:eigen_sharp}
    Note that the conclusion of Theorem \ref{thm:eigen_result} and its proof can also be applied to continuous Galerkin methods on surfaces.
    We notice here that the geometric approximation error and the function discretization error are not of the same order in the higher-order element cases for the eigenvalue approximation, i.e., the convergence rates of order $\mathcal{O}(h^{\min\set{2l,2k}})$, which was observed in some numerical examples in \cite{BonDemOwe18}. 
    The proven error bounds turn out to be theoretically sharp given the geometric approximation error and the function approximation error there.
    Here we provide an example which shows that the error contribution from geometric discrepancy cannot be improved in general cases. 
    We pick simply the exact surface to be a unit sphere $\Gamma$. Let the approximating manifold be another sphere $\hat{\Gamma}_h^k$ with radius $1+h^{k+1}$, which satisfies the geometric approximation error $\norm{\pi -\pi_h^{k}}\leq h^{k+1}$. Now we consider the Laplace-Beltrami operator on $\hat{\Gamma}_h^k$ and its eigenvalues.
    Using the analytical formula of spherical harmonics, we know that the eigenvalues associated to $\hat{\Gamma}_h^k$ are of the following form $\lambda_h^k\in \left(\frac{n(n+1)}{(1+h^{k+1})}\right)_{n=1}^\infty$, while for $\Gamma$ are $\lambda\in (n(n+1))_{n=1}^\infty$.
    The error of eigenvalues is $\abs{\lambda-\lambda_h^k}=\frac{n(n+1)h^{k+1} }{(1+h^{k+1})}$. This shows even in this ideal case, the eigenvalues' approximation errors of order $h^{2k}$ cannot be achieved, when the geometric approximation error is of order $\mathcal{O}(h^{k+1})$ for $k>1$.
 Note that in this counter example, the approximating surface is not reconstructed from an algorithm, however it satisfies the geometric error conditions. Thus it is valid for testing the optimality of the approximation error in the eigenvalue problem given the geometric approximation assumptions.
\end{remark}

	\section{Numerical results}
	\label{sec:numerical}
	
In the following, we present a series of benchmark numerical experiments to verify our theoretical findings on the error analysis of the proposed DG method on surfaces reconstructed from point clouds. We pick two representative manifolds, the unit sphere and a torus surface, which have genus zero and one, respectively.  
Before going to numerical results of the optimal convergence rates of the DG method for solving both source and eigenvalue problems of LB operator on point clouds, we first provide a validation of the geometric error bounds of Algorithm \ref{alg:point_estimate}.

	\subsection{Geometric error}
	To simplify the presentation, we introduce the following notations 
	\begin{equation*}
	\begin{split}
	&e^k = \max_{i \in N} |\psi^i-\xi^i|, \quad\quad \text{ and }\quad\quad\,\, \,
	e_t^k = \max_{i \in N} |\psi_{\Omega_h}^i-\xi_{\Omega_h}^i|,
	\end{split}
	\end{equation*}
	where the superscript $k$ denotes the $k$-th order polynomials used in the patch reconstruction from point cloud, the subscript $t$ denotes the vector component in $\Omega_{h}$ (the projection of the vector onto the plane $\Omega_{h}$), and $N$ represents the total number of nodal points of the all patches in $\Omega_h$.

	We first check the accuracy claimed for Algorithm \ref{alg:point_estimate} based on point cloud which is sampled from the unit sphere and is uniform distributed on its spherical coordinate. In the first test, the initial mesh is generated using CGAL\cite{AlSG2020} whose vertices are $\mathcal{O}(h^2)$ to the unit sphere.   In practice, such meshes can be generated using Poisson surface reconstruction \cite{KaBH2006} which is available in the computational geometry algorithm library CGAL \cite{AlSG2020}. To depict the convergence rates, we conduct a uniform refinement for each triangle by dividing it into four similar subtriangles. A main difference between the classical refinement for surface finite element methods \cite{Demlow2009} is that we do not ask for the precise embedding of submanifolds, e.g., the level set functions or parameterization functions. Note that Assumption \ref{assum:geometry} holds already due to the interpolation of the mesh in this example. For the general case, we can adjust the vertices of the mesh using the point cloud reconstruction before proceeding with uniform refinements, and this costs no extra work.

To evaluate the implications of Proposition \ref{prop:nodal_accu}, we compare the reconstructed nodal points with the exact nodal ones, as depicted in Table \ref{tab:spherenormal}. Given the spherical structure, obtaining the exact nodal points involves stretching the reconstructed nodes along radial directions. The observed errors generally align with the $\mathcal{O}(h^{k+1})$ error bounds, except for some outliers. Additionally, we present errors after projecting paired nodal points onto the parametric domain $\Omega_h$ in Table \ref{tab:spheretang}, illustrating a $\mathcal{O}(h^{k+2})$ convergence rate. These findings are in accordance with the theoretical results established in Proposition \ref{prop:nodal_accu}.
	
	\begin{table}[htb!]
		\centering
		\caption{Errors of patch reconstruction from point cloud on unit sphere}
		\label{tab:spherenormal}
		\resizebox{\textwidth}{!}{
			\begin{tabular}{|c|c|c|c|c|c|c|c|c|c|c|}
				\hline 
				$N_v$&$e^1$&Order&$e^2$&Order&$e^3$&Order&$e^4$&Order\\ 
				\hline
				222&1.12e-02&--&1.52e-05&--&1.35e-05&--&9.02e-08&--\\ \hline
				882&2.95e-03&1.94&1.09e-05&0.47&2.68e-06&2.34&4.27e-08&1.09\\ \hline
				3522&7.21e-04&2.03&7.14e-07&3.94&9.42e-08&4.84&3.40e-10&6.98\\ \hline
				14082&1.80e-04&2.00&4.65e-08&3.94&4.83e-09&4.29&3.56e-12&6.58\\ \hline
				56322&4.64e-05&1.96&3.77e-09&3.62&4.07e-10&3.57&1.22e-13&4.87\\ \hline
		\end{tabular}}
		\caption{Error component in the parametric domain for sphere point cloud}
		\label{tab:spheretang}
		\resizebox{\textwidth}{!}{
			\begin{tabular}{|c|c|c|c|c|c|c|c|c|c|c|}
				\hline 
				$N_v$&$e_t^1$&Order&$e_t^2$&Order&$e_t^3$&Order&$e_t^4$&Order\\ 
				\hline	
				222&3.88e-04&--&1.63e-05&--&1.16e-05&--&1.44e-07&--\\ \hline
				882&1.71e-03&-2.15&1.81e-05&-0.16&5.19e-06&1.16&1.15e-07&0.33\\ \hline
				3522&2.13e-04&3.01&5.92e-07&4.94&1.61e-07&5.01&9.09e-10&6.99\\ \hline
				14082&2.69e-05&2.98&1.91e-08&4.96&5.41e-09&4.90&8.15e-12&6.80\\ \hline
				56322&3.47e-06&2.96&7.04e-10&4.76&2.37e-10&4.51&9.61e-14&6.41\\ \hline
		\end{tabular}}
	\end{table}
	
	Notice that, in the sphere case, surprisingly, we can observe some geometric superconvergence phenomenon in the reconstructed patches when the even-order polynomial is employed like quadratic and quartic polynomials.  This implies the superconvergence geometric approximation results in the eigenvalue problem later.

	Another set of tests are implemented with point cloud presentation of a torus surface:
	\begin{equation}\label{equ:torus}
	\left\{
	\begin{array}{l}
	x = (4\cos(\theta) + 1)\cos(\phi),\\
	y = (4\cos(\theta)+1)\sin(\phi), \\
	z = 4\sin(\theta).	
	\end{array}
	\right.
	\end{equation}
	The point cloud is generated from uniformly distributed points in $(\theta, \phi)$ coordinate and then projected onto the torus.

In Table \ref{tab:torusproj}, we present the numerical error for the reconstructed nodal points as the first example. We observe error bounds of order $\mathcal{O}(h^{k+1})$ when using $k$-th order polynomials. Similarly, the error bounds of $\mathcal{O}(h^{k+2})$ for the nodal pairs projected onto the parameter domain align with Proposition \ref{prop:nodal_accu}. Unlike the unit sphere case, there seems to be no geometrical superconvergence for even polynomials, which differs from the previous example.

	\begin{table}[htb!]
		\centering
		\footnotesize
		\caption{Errors of patch reconstruction from point cloud on torus}
		\label{tab:torusproj}
		\resizebox{\textwidth}{!}{
			\begin{tabular}{|c|c|c|c|c|c|c|c|c|c|c|}
				\hline 
				$N_v$&$e^1$&Order&$e^2$&Order&$e^3$&Order&$e^4$&Order\\ 
				\hline
				200&4.67e-02&--&1.66e-03&--&3.17e-04&--&6.63e-05&--\\ \hline
				800&1.44e-02&1.69&7.79e-04&1.09&8.62e-05&1.88&9.45e-06&2.81\\ \hline
				3200&3.39e-03&2.09&5.81e-05&3.75&2.54e-06&5.08&1.98e-07&5.58\\ \hline
				12800&9.00e-04&1.91&5.81e-06&3.32&1.15e-07&4.47&3.92e-09&5.66\\ \hline
				51200&2.48e-04&1.86&1.09e-06&2.42&7.06e-09&4.02&1.14e-10&5.11\\ \hline
		\end{tabular}}
		\caption{Error component in the parametric domain for point cloud on torus}
		\label{tab:torus_tang}
		\resizebox{\textwidth}{!}{
			\begin{tabular}{|c|c|c|c|c|c|c|c|c|c|c|}
				\hline 
				$N_v$&$e_t^1$&Order&$e_t^2$&Order&$e_t^3$&Order&$e_t^4$&Order\\ 
				\hline	
				200&2.92e-03&--&1.37e-03&--&3.48e-04&--&7.41e-05&--\\ \hline
				800&1.67e-02&-2.52&1.85e-03&-0.43&2.22e-04&0.65&4.04e-05&0.88\\ \hline
				3200&2.46e-03&2.77&1.25e-04&3.89&7.83e-06&4.82&6.07e-07&6.06\\ \hline
				12800&3.17e-04&2.95&8.34e-06&3.90&2.46e-07&4.99&1.32e-08&5.53\\ \hline
				51200&4.34e-05&2.87&6.82e-07&3.61&8.29e-09&4.89&2.42e-10&5.76\\ \hline
		\end{tabular}}
	\end{table}

	\subsection{PDEs on patches reconstructed from a point cloud of the unit sphere}
	\label{ssec:sphere}
In this test, we first give results for the Laplace-Beltrami equation  \eqref{eq:Lap_Bel}.  The function $f(x)$ on the right hand side is chosen to fit the exact solution $u(x) = \sin(x_1)\sin(x_2)\sin(x_3)$. We choose the finite element degree $l$ and the geometric degree $k$ to be equal to some $p$ for $p \in {1, 2, 3, 4}$. According to Theorem \ref{thm:convergence}, we would expect $\mathcal{O}(\max{h^{l+1},h^{k+1}})$ for $L_2$ error and $\mathcal{O}(\max{h^{l},h^{k}})$ convergence for $H^1$ error for sufficiently small $h$. To avoid confusion, we denote the $L^2$ error (or $H^1$ error) with $p$th order degree as $e^p$ (or $De^p$). The numerical results are plotted in Figure \ref{fig:sphere_fem}, where optimal convergence rates can be observed as expected.
	\begin{figure}[!h]
		\centering
		%\subcaptionbox{\label{fig:sphere_l2err}}
		{\includegraphics[width=0.47\textwidth]{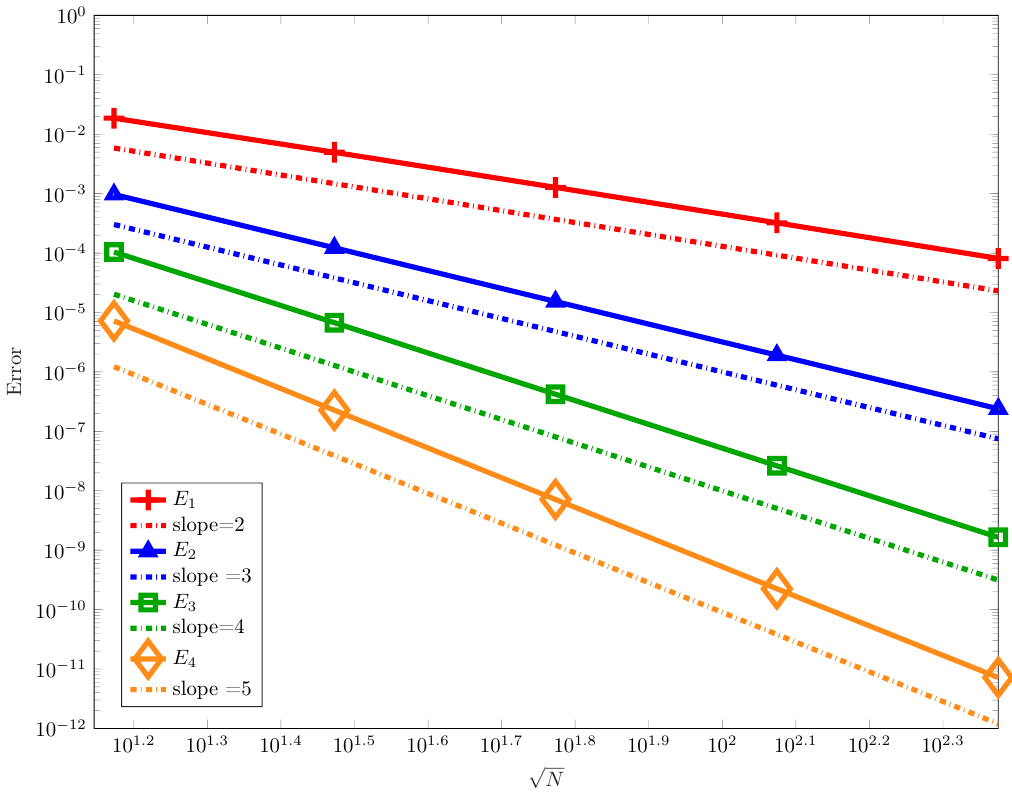}}
		\vspace{0.1in}
		%\subcaptionbox{\label{fig:sphere_h1err}}
		{\includegraphics[width=0.47\textwidth]{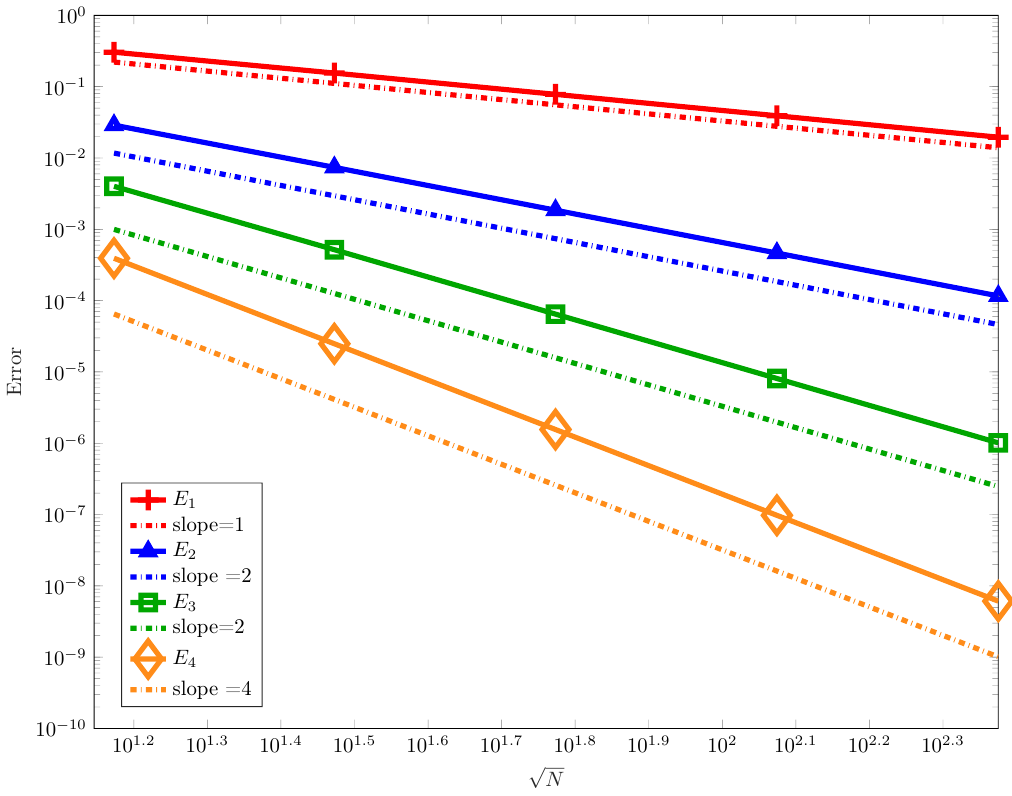}}
		\caption{(Color online) Numerical results of Laplace-Beltrami equation on the unit sphere. Left (a): $L_2$ error; Right (b): $H_1$ error.}
		\label{fig:sphere_fem}
	\end{figure}
	
We proceed by testing the DG discretization of the Laplace-Beltrami eigenvalue problem \eqref{eq:eigen_Lap_Bel} using the same surface patches. For the Laplace-Beltrami operator on the unit sphere, both eigenvalues and eigenfunctions are explicitly known. In these experiments, our focus lies on computing the first nonzero eigenvalue $\lambda_2=\lambda_3=\lambda_4 = 2$. These eigenfunctions, namely $u_i = x_{i-1}$ for $i = 2, 3, 4$, are known as spherical harmonics \cite{AbSt1964}.

	As per Theorem \ref{thm:eigen_result}, the eigenvalue $\lambda_i$ achieves optimal convergence at a rate of $\mathcal{O}(\max\set{h^{2l},h^{k+1}})$. To exemplify, for $l=2$, we require $k=3$ to observe this optimal convergence. The numerical outcomes corresponding to these parameters are presented in Table \ref{tab:sphere_eigen_k2l3}, confirming the anticipated convergence rates.

	\begin{table}[!htb]
		\centering
		\caption{Numerical result of Laplace-Beltrami eigenvalue problem on the unit sphere point cloud when $l=2$ and $k= 3$}
		\label{tab:sphere_eigen_k2l3}
		\resizebox{\textwidth}{!}{
			\begin{tabular}{|c|c|c|c|c|c|c|c|}
				\hline 
				i&$N_v$ & $|\lambda_i-\lambda_{i,h}|$ & Order & $\|u_i-u_{i,h}\|_{0,\Gamma_h}$&Order & $|u_i-u_{i,h}|_{1,\Gamma_h}$&Order
				\\ \hline \hline
				2 & 222 & 3.99e-05 & -- & 4.20e-04 & -- & 1.07e-02 & -- \\ \hline
				2 & 882 & 2.82e-06 & 3.84 & 5.26e-05 & 3.01 & 2.72e-03 & 1.99 \\ \hline
				2 & 3522 & 1.81e-07 & 3.97 & 6.59e-06 & 3.00 & 6.83e-04 & 1.99 \\ \hline
				2 & 14082 & 1.16e-08 & 3.97 & 8.27e-07 & 3.00 & 1.72e-04 & 1.99 \\ \hline
				3 & 222 & 5.18e-05 & -- & 4.28e-04 & -- & 1.09e-02 & -- \\ \hline
				3 & 882 & 3.50e-06 & 3.91 & 5.37e-05 & 3.01 & 2.76e-03 & 1.99 \\ \hline
				3 & 3522 & 2.25e-07 & 3.97 & 6.74e-06 & 3.00 & 6.94e-04 & 1.99 \\ \hline
				3 & 14082 & 1.43e-08 & 3.98 & 8.45e-07 & 3.00 & 1.74e-04 & 1.99 \\ \hline
				4 & 222 & 6.52e-05 & -- & 4.45e-04 & -- & 1.12e-02 & -- \\ \hline
				4 & 882 & 4.38e-06 & 3.92 & 5.58e-05 & 3.01 & 2.85e-03 & 1.99 \\ \hline
				4 & 3522 & 2.79e-07 & 3.97 & 6.98e-06 & 3.00 & 7.17e-04 & 1.99 \\ \hline
				4 & 14082 & 1.77e-08 & 3.98 & 8.75e-07 & 3.00 & 1.80e-04 & 1.99 \\ \hline
		\end{tabular}}
	\end{table}

	\begin{table}[!htb]
		\centering
		\caption{Numerical result of Laplace-Beltrami eigenvalue problem on the unit sphere point cloud when $l=3$ and $k= 4$}
		\label{tab:sphere_eigen_k3l4}
		
		\resizebox{\textwidth}{!}{
			\begin{tabular}{|c|c|c|c|c|c|c|c|}
				\hline 
				i&$N_v$ & $|\lambda_i-\lambda_{i,h}|$ & Order & $\|u_i-u_{i,h}\|_{0,\Gamma_h}$&Order & $|u_i-u_{i,h}|_{1,\Gamma_h}$&Order
				\\ \hline\hline
				2 & 222 & 1.46e-08 & -- & 2.00e-05 & -- & 7.08e-04 & -- \\ \hline
				2 & 882 & 8.98e-10 & 4.04 & 1.33e-06 & 3.92 & 9.35e-05 & 2.93 \\ \hline
				2 & 3522 & 1.59e-11 & 5.82 & 8.01e-08 & 4.06 & 1.14e-05 & 3.04 \\ \hline
				2 & 14082 & 1.69e-11 & -0.08 & 4.59e-09 & 4.13 & 1.34e-06 & 3.09 \\ \hline
				3 & 222 & 3.15e-08 & -- & 1.97e-05 & -- & 6.95e-04 & -- \\ \hline
				3 & 882 & 1.21e-09 & 4.72 & 1.19e-06 & 4.07 & 8.52e-05 & 3.04 \\ \hline
				3 & 3522 & 1.73e-11 & 6.14 & 7.84e-08 & 3.93 & 1.11e-05 & 2.94 \\ \hline
				3 & 14082 & 2.10e-11 & -0.28 & 4.72e-09 & 4.06 & 1.36e-06 & 3.03 \\ \hline
				4 & 222 & 4.84e-08 & -- & 1.77e-05 & -- & 6.39e-04 & -- \\ \hline
				4 & 882 & 1.48e-09 & 5.06 & 1.22e-06 & 3.87 & 8.75e-05 & 2.88 \\ \hline
				4 & 3522 & 1.88e-11 & 6.31 & 7.12e-08 & 4.11 & 1.04e-05 & 3.08 \\ \hline
				4 & 14082 & 3.30e-11 & -0.82 & 4.97e-09 & 3.84 & 1.41e-06 & 2.88 \\ \hline 
		\end{tabular}}
	\end{table}

	\begin{table}[!htb]
		\centering
		\caption{Numerical result of Laplace-Beltrami eigenvalue problem on the unit sphere when $k=2$ and $l = 2$}
		\label{tab:sphere_eigen_k2l2}
		\resizebox{\textwidth}{!}{
			\begin{tabular}{|c|c|c|c|c|c|c|c|}
				\hline 
				i&$N_v$ & $|\lambda_i-\lambda_{i,h}|$ & Order & $\|u_i-u_{i,h}\|_{0,\Gamma_h}$&Order & $|u_i-u_{i,h}|_{1,\Gamma_h}$&Order
				\\ \hline \hline
				2 & 222 & 1.09e-04 & -- & 2.12e-05 & -- & 1.21e-03 & -- \\ \hline
				2 & 882 & 7.65e-06 & 3.86 & 3.58e-06 & 2.58 & 4.83e-04 & 1.33 \\ \hline
				2 & 3522 & 4.67e-07 & 4.04 & 2.16e-07 & 4.06 & 6.08e-05 & 3.00 \\ \hline
				2 & 14082 & 2.92e-08 & 4.00 & 1.34e-08 & 4.01 & 7.62e-06 & 3.00 \\ \hline
				3 & 222 & 1.22e-04 & -- & 2.42e-05 & -- & 1.20e-03 & -- \\ \hline
				3 & 882 & 8.28e-06 & 3.90 & 3.86e-06 & 2.66 & 4.90e-04 & 1.30 \\ \hline
				3 & 3522 & 5.14e-07 & 4.02 & 2.35e-07 & 4.04 & 6.16e-05 & 3.00 \\ \hline
				3 & 14082 & 3.22e-08 & 4.00 & 1.46e-08 & 4.00 & 7.72e-06 & 3.00 \\ \hline
				4 & 222 & 1.44e-04 & -- & 2.77e-05 & -- & 1.18e-03 & -- \\ \hline
				4 & 882 & 9.94e-06 & 3.88 & 4.58e-06 & 2.61 & 4.99e-04 & 1.25 \\ \hline
				4 & 3522 & 6.11e-07 & 4.03 & 2.74e-07 & 4.07 & 6.27e-05 & 3.00 \\ \hline
				4 & 14082 & 3.81e-08 & 4.00 & 1.69e-08 & 4.02 & 7.86e-06 & 3.00 \\  \hline
		\end{tabular}}
	\end{table}

	We also explore another choice of $k$ and $l$. The numerical outcomes for $l=3$ and $k=4$ are detailed in Table \ref{tab:sphere_eigen_k3l4}. According to Theorem \ref{thm:eigen_result}, we would expect a convergence order of $\mathcal{O}(h^5)$ for the eigenvalues. However, remarkably, we observe a convergence of $\mathcal{O}(h^6)$ for the eigenvalue, attributed to the superconvergence observed in the geometric approximation as previously reported.
To further illustrate this, we examine the scenario where $k=l=2$. The numerical results are presented in Table \ref{tab:sphere_eigen_k2l2}. Once again, we observe improved convergence rates, supporting the error bounds of $\mathcal{O}(\max\set{h^{2l},h^{k+2}})$ for even $k$.
Surprisingly, a $\mathcal{O}(h^{k+2})$ (or $\mathcal{O}(h^{k+1})$) superconvergence is also noted in the approximation of eigenfunctions in $L^2$ norm (or $H^1$ norm).
Although, as highlighted in Remark \ref{rem:eigen_sharp}, a sharper estimation for eigenvalue approximation is not generally available given the current geometric approximation conditions, it is worth exploring specific geometric conditions which might lead to such superconvergence. This could be an interesting topic for future research.

	\subsection{PDEs on patches reconstructed from a point cloud of a torus}
	\label{ssec:torus}
	
	Numerical results of the DG method on the patches reconstructed from the torus point cloud are collected here. Similar to the last example, we start with the Laplace-Beltrami equation \eqref{eq:Lap_Bel}. The cooked up exact solution is $u(x) = x_1-x_2$, and then the right-hand side function $f(x)$ can be simply computed. We set $k = l = p$ for some $p=1,\ldots,4$. The numerical errors are documented in Figure \ref{fig:torus_fem}. From there, we can spot the optimal convergence rates for both $L^2$ and $H^1$ errors.
	
	\begin{figure}[!h]
		\centering
		%\subcaptionbox{\label{fig:torus_l2err}}
		{\includegraphics[width=0.42\textwidth]{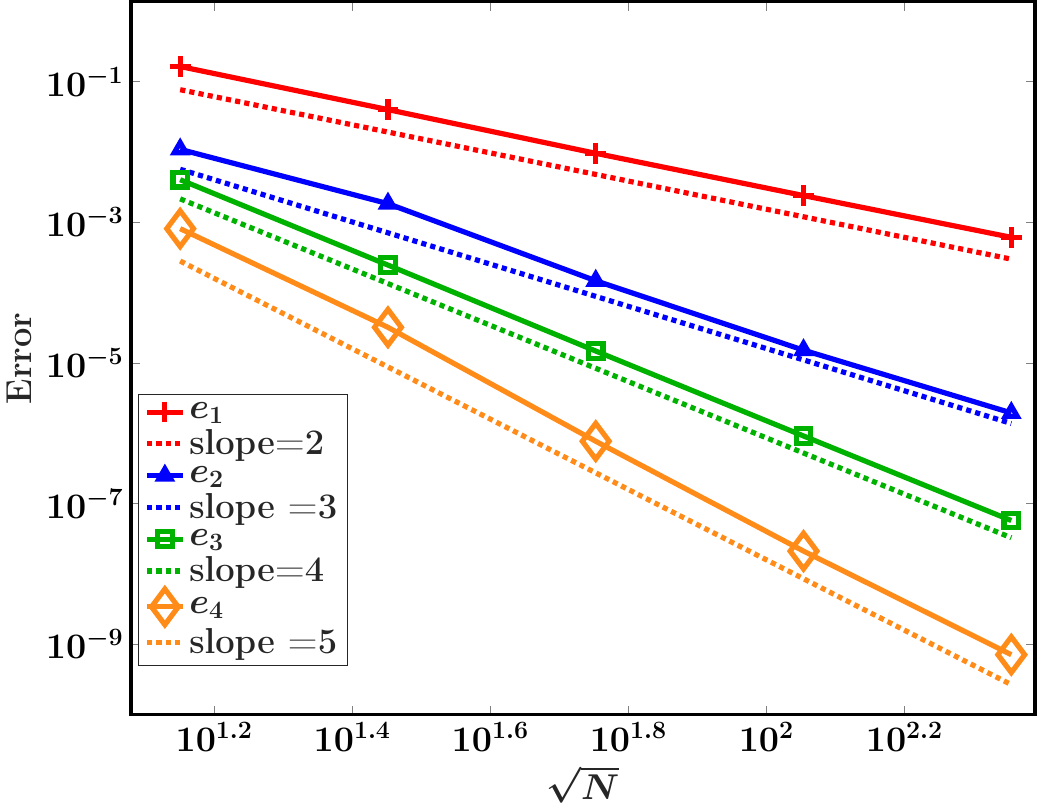}}
		\vspace{0.1in}
		%\subcaptionbox{\label{fig:torus_h1err}}
		{\includegraphics[width=0.42\textwidth]{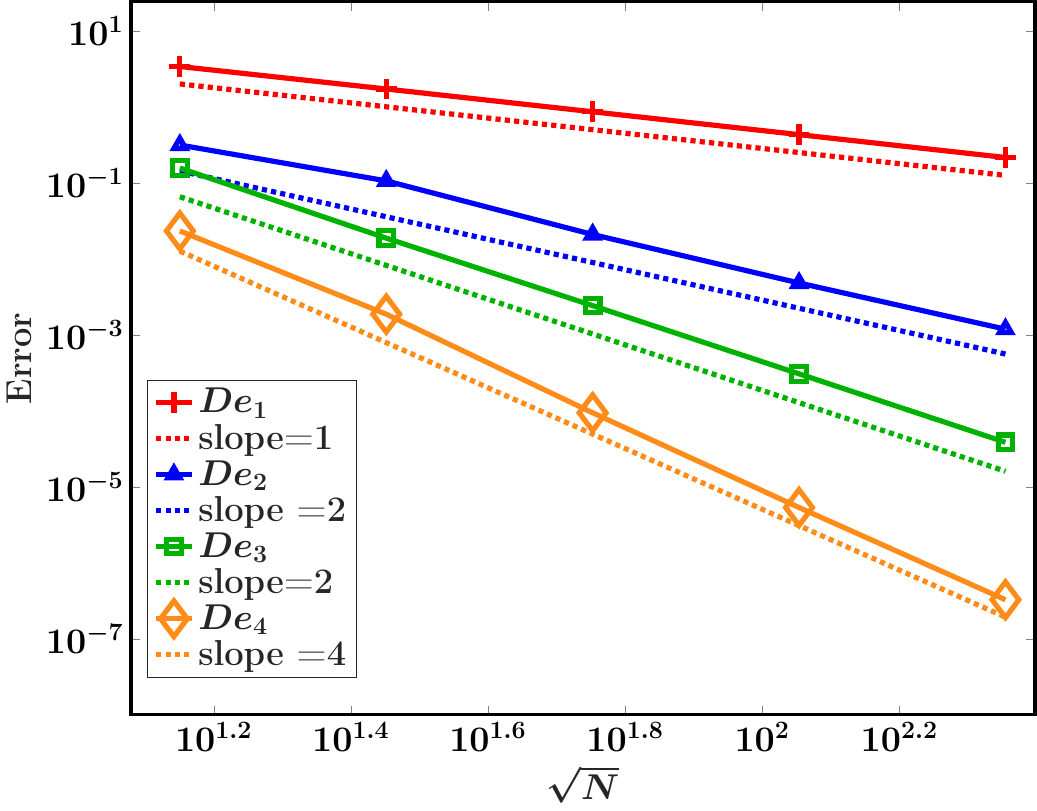}}
		\caption{(Color online) Numerical results of Laplace-Beltrami equation on torus point could. Left (a): $L_2$ error; Right (b): $H_1$ error.}
		\label{fig:torus_fem}
	\end{figure}

	\begin{figure}[!h]
		\centering
		%\subcaptionbox{\label{fig:torus_eigen2}}
		{\includegraphics[width=0.42\textwidth]{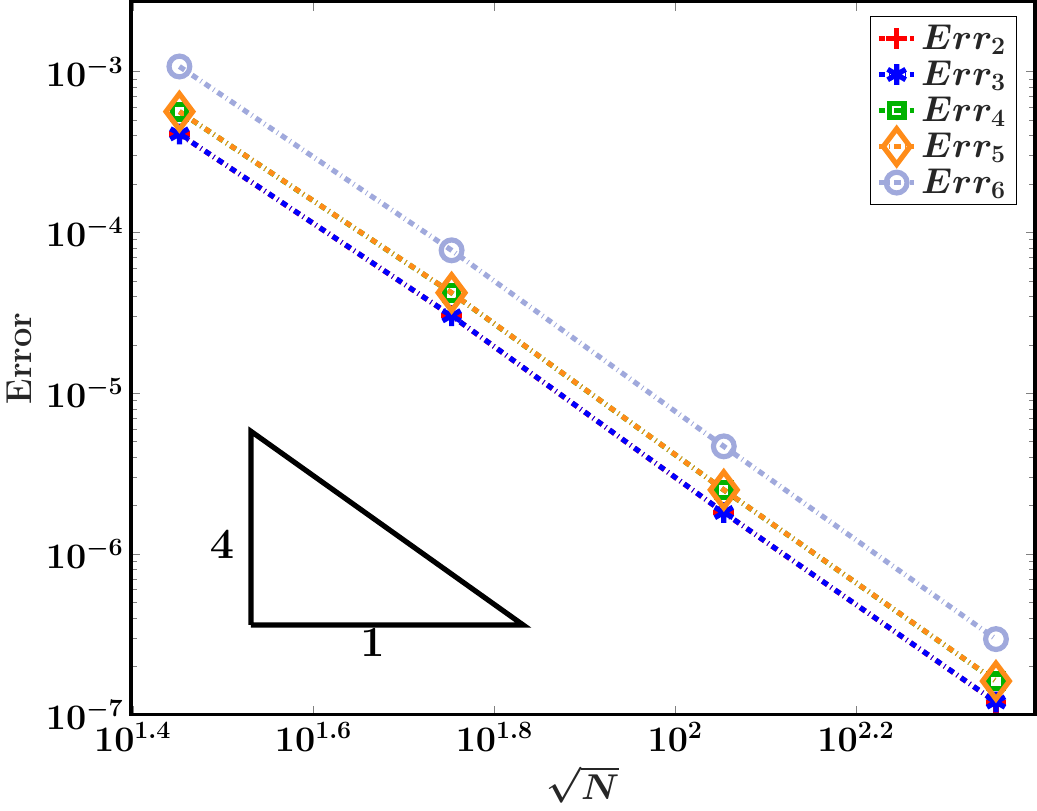}}
		\vspace{0.1in}
		%\subcaptionbox{\label{fig:torus_eigen4}}
		{\includegraphics[width=0.42\textwidth]{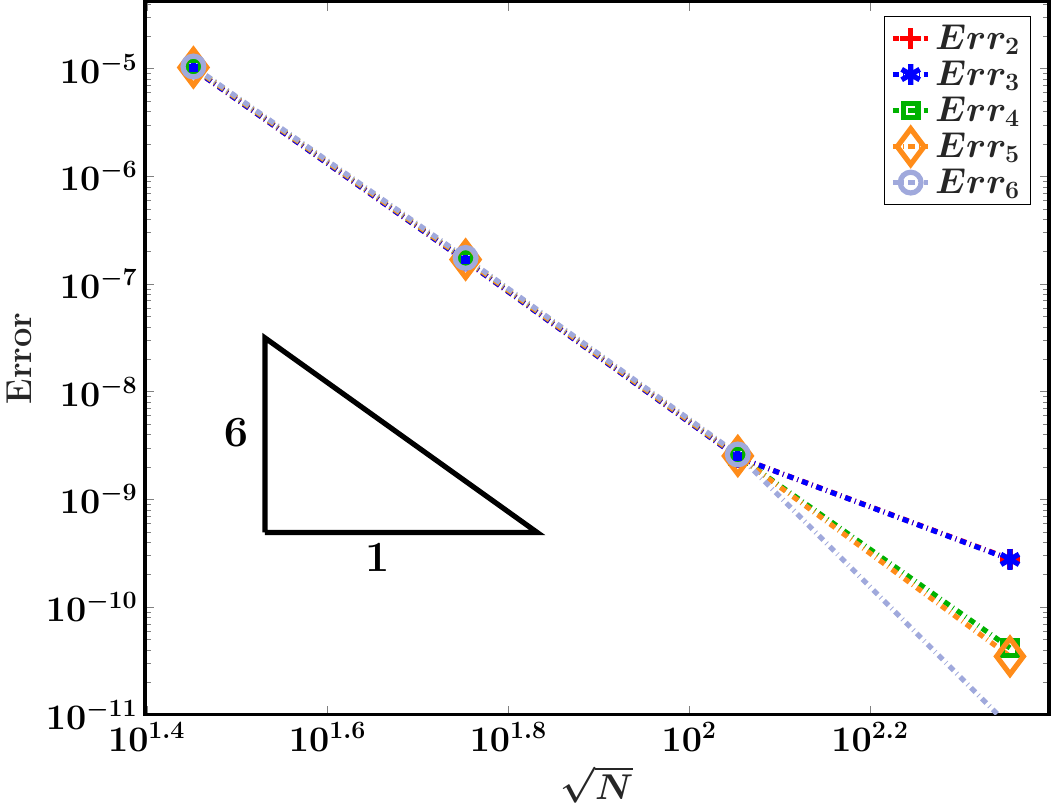}}
		\caption{(Color online) Numerical results of Laplace-Beltrami eigenvalue problem on torus point cloud. Left (a): $k=l=2$; Right (b): $k=l=4$.}
		\label{fig:torus_eigen}
	\end{figure}
	
	For the comparison of eigenvalues for the Laplace-Beltrami operator, it is different to the unit sphere case, as the exact eigenvalues of the Laplace-Beltrami operator on a torus are not known to us. 
To study the convergence rates, we use the following quotient
\begin{equation}
    Err_i = \frac{|\lambda_{i, h_{j+1}}-\lambda_{i, h_{j}}|}
    {\lambda_{i, h_{j+1}}},
\end{equation}
where $h_{j}$ is the mesh-size on the $j$th level of meshes. The numerical results are reported in Figure \ref{fig:torus_eigen} Left (a) for $k=l=2$. We compare the first five nontrivial eigenvalues.
What stands out in this figure is that $4$th order convergence is shown up, which seems indicate $\mathcal{O}(\max\set{h^{2l},h^{k+2}})$ bounds for eigenvalue approximation when $k$ is even.  
In Figure \ref{fig:torus_eigen} Right (b), we display the eigenvalue approximation error when $k=l=4$. Superconvergence of the geometric approximations are presented again in both the even-order polynomial cases.
This kind of superconvergence results may deserve further investigation.

\section{Conclusion}\label{sec:con}
In this paper, we have proposed a high-order DG method for numerical solutions of PDEs on point clouds. Our particular focus is on the numerical analysis of this method since existing results in the literature are not applicable in such a setting. To do so, we introduced a new error analysis framework for geometric approximation without exact geometric information and global continuity assumptions. This framework is not only suitable for the method proposed in this paper but also able to accommodate existing methods like surface Galerkin methods (including both continuous and discontinuous finite elements) in the literature, even though the numerical tests are implemented for a specific DG method only.

\subsection*{Acknowledgment}
The authors acknowledge the anonymous reviewer for very careful reading and suggestions which improved the presentation of the paper.
The authors thank Prof. Buyang Li (Hong Kong PolyU) for valuable comments which improved the paper as well.
Guozhi Dong was supported by the NSFC Grants No. 12001194, No. 12471402, and the NSF grant of Hunan Province No. 2024JJ5413.
Hailong Guo was partially supported by the Andrew Sisson Fund, Dyason Fellowship, and the Faculty Science Researcher
Development Grant of the University of Melbourne. Zuoqiang Shi was supported by the NSFC Grant 12071244.
{	\appendix
	\section{Proof of Lemma \ref{lem:aux_appr}}
	\label{app:approx}
	\begin{proof}
		We consider first the case that $p_h^{k,j}$ is produced from interpolation by the polynomial graph. Let us denote by $p_h^{*k,j}$. This is when the number $m$ for calculating \eqref{eq:local_patch_appro} in Algorithm \ref{alg:point_estimate} equals to the number of unknown coefficients for determining the polynomial, let us denote it by $m*$.
		Then the above estimate is a direct result from approximation theory. When $ m>m*$, we use the property that $p_h^{k,j}$ is a minimizer of \eqref{eq:local_patch_appro}. It is evident that
		\[ \sum_{i=1}^m\abs{p^j(r_i^j)-p_h^{k,j}(r_i^j)}^2\leq \sum_{i=1}^m\abs{p^j(r_i^j)-p_h^{*k,j}(r_i^j)}^2, \]
		since $p^j(r_i^j)=\zeta_i^j$ for $i=1,\cdots,m$.
		We have $\abs{p^j(r_i^j)-p_h^{*k,j}(r_i^j)}\leq Ch^{k+1}$ for all $i=1,\cdots,m$, which indicates that
		\[ \abs{p^j(r_i^j)-p_h^{k,j}(r_i^j)}\leq C'h^{k+1} \text{ for all } \; i=1,\cdots,m.\]
		We have also that $p_h^{*k,j}$ interpolates $p^j$.
		As $p_h^{*k,j}$ and $p_h^{k,j}$ both be $k^{th}$ order local polynomials. Then we can pick $m*$ points out of the $m$ points, and to have $p_h^{k,j}=\sum_i^{m*} p_h^{k,j}(r_i^j)\phi_i$, and $p_h^{*k,j}=\sum_i^{m*} p^{j}(r_i^j)\phi_i$, where $(\phi_i)_{i=1}^{m*}$ are Lagrange polynomial bases. Since $\abs{p_h^{*k,j}(r_i^j)-p_h^{*k,j}(r_i^j)}\leq Ch^{k+1}$, and $\abs{\phi_i}\leq C$ due to our assumption on the distribution of the selected reconstruction points. We have then  for all $r\in \Omega_h^j$
		\[ \abs{p_h^{*k,j}(r) - p_h^{k,j}(r)}
		%	=\abs{\sum_{i=1}^{m*} (p_h^{k,j}(r_i^j) -p_h^{*k,j}(r_i^j))\phi_i(r)}
		\leq \sum_{i=1}^{m*}\abs{\phi_i(r)}\abs{p_h^{*k,j}(r_i^j) -p_h^{k,j}(r_i^j)}\leq Ch^{k+1} .\]
		This gives us the estimate using the triangle inequality 
		\[\abs{p^j(r)-p_h^{k,j}(r)}\leq \abs{p^j(r)-p_h^{*k,j}(r)}+\abs{p_h^{*k,j}(r) -  p_h^{k,j}(r)} =\mathcal{O}( h^{k+1})\quad \text{ for all } r\in \Omega_h^j.\]
		Finally, standard argument leads to the estimate on the derivatives (gradient):
		\[\abs{\nabla p^j(r)- \nabla p_h^{k,j}(r)} =\mathcal{O}( h^{k}) \quad \text{ for all } r\in \Omega_h^j.\]
	\end{proof}
	
	\section{Details on the estimate for each terms in \eqref{eq:R_error}}
	\label{app:DG_bilinear}
	Note that the results of Theorem \ref{thm:geo_error} are crucial for the estimates here.
	%\begin{equation*}
	%	\begin{aligned}
	%		&	\mathcal{R}(v_h)=\\
	%		&\sum_j  \int_{\Omega_h^j} (\nabla \bar{u}^{k,l}_h)^\top \left((g^k_h)^{-1}-g^{-1}\right) \nabla \bar{v}_h \sqrt{\abs{g^k_h}} + (\nabla \bar{u}^{k,l}_h)^\top g^{-1}\nabla \bar{v}_h \left(\sqrt{\abs{g^k_h}} - \sqrt{\abs{g}}\right) dA   \\
	%		& -  \sum_{\bar{e}_h\in \bar{\mathcal{E}}_h}\int_{ \bar{e}_h} \{(\nabla \bar{u}^{k,l}_h)^\top \left( (g^k_h)^{-1}  (\partial \pi_h^k)^\top n_h - g^{-1}(\partial \pi )^\top n\right) \}[\bar{v}_h l_{g^k_h} ]dE \\
	%		&  -  \sum_{\bar{e}_h\in \bar{\mathcal{E}}_h}\int_{ \bar{e}_h} \{(\nabla \bar{v}_h)^\top \left( (g^k_h)^{-1}  (\partial \pi_h^k)^\top n_h - g^{-1}(\partial \pi )^\top n\right) \}[\bar{u}^{k,l}_h l_{g^k_h} ]dE \\
	%		& - \sum_{\bar{e}_h\in \bar{\mathcal{E}}_h} \int_{ \bar{e}_h} \{(\nabla \bar{u}^{k,l}_h)^\top g^{-1}(\partial \pi )^\top n \}[\bar{v}_h\left(l_{g^k_h}-l_g\right) ]  dE \\
	%		&  - \sum_{\bar{e}_h\in \bar{\mathcal{E}}_h} \int_{ \bar{e}_h} \{(\nabla \bar{v}_h)^\top g^{-1}(\partial \pi )^\top n \}[\bar{u}^{k,l}_h\left(l_{g^k_h}-l_g\right) ]  dE \\
	%		& + \beta_h \sum_{\bar{e}_h\in \bar{\mathcal{E}}_h}  \int_{\bar{e}_h} [\bar{u}^{k,l}_h] [\bar{v}_h]\{\left( l_{g^k_h}-l_g\right)\} dE
	%		+  \int_{\Omega_h^j}  \bar{f}   \bar{v}_h  \left(\sqrt{\abs{g}} - \sqrt{\abs{g^k_h}}\right) dA.
	%	\end{aligned}
	%\end{equation*}
	For the first term in \eqref{eq:R_error}, we have:
	\begin{equation*}
	\begin{aligned}
	&	\sum_j  \int_{\Omega_h^j} (\nabla \bar{u}^{k,l}_h)^\top \left((g^k_h)^{-1}-g^{-1}\right) \nabla \bar{v}_h \sqrt{\abs{g^k_h}} + (\nabla \bar{u}^{k,l}_h)^\top g^{-1}\nabla \bar{v}_h \left(\sqrt{\abs{g^k_h}} - \sqrt{\abs{g}}\right) dA   \\
	&\leq \norm{g^k_h\left((g^k_h)^{-1}-g^{-1}\right) }_{L^\infty} \abs{\sum_j \int_{\Omega_h^j} (\nabla \bar{u}^{k,l}_h)^\top	(g^k_h)^{-1}\nabla \bar{v}_h \sqrt{\abs{g^k_h}} dA}+\\
	&\;\;\;\; \norm{\left(\sqrt{\abs{g^k_h}} - \sqrt{\abs{g}}\right)\sqrt{\abs{g}}^{-1}}_{L^\infty} \abs{\sum_j \int_{\Omega_h^j} (\nabla \bar{u}^{k,l}_h)^\top	g^{-1}\nabla \bar{v}_h \sqrt{\abs{g}} dA}\\
	&\leq Ch^{k+1}\left(\norm{u_h^{k,l}}_{H^1(\hat{\Gamma}_h^{k})}\norm{(T_h^k)^{-1}v_h}_{H^1(\hat{\Gamma}_h^{k})} +\norm{T_h^ku_h^{k,l}}_{H^1(\Gamma_h)}\norm{v_h}_{H^1(\Gamma_h)} \right)\\
	&=\mathcal{O}(h^{k+1})\norm{T_h^ku_h^{k,l}}_{V_h^l}\norm{v_h}_{V_h^l} ,
	\end{aligned}
	\end{equation*}
	where the last inequality is due to the equivalence of the norms of $H^1(\hat{\Gamma}_h^{k})$ and $H^1(\Gamma_h)$ and the embedding that $H^1(\Gamma_h)\subset V_h^l$.
	For the second term, we have
	\begin{equation*}
	\begin{aligned}
	&-  \sum_{\bar{e}_h\in \bar{\mathcal{E}}_h}\int_{ \bar{e}_h} \{(\nabla \bar{u}^{k,l}_h)^\top \left( (g^k_h)^{-1}  (\partial \pi_h^k)^\top n_h - g^{-1}(\partial \pi )^\top n\right) l_{g^k_h}\}[\bar{v}_h  ]dE \\
	\leq & \norm{\frac{l_{g_h^k}}{l_g}}_{L^\infty} \norm{\left( (g^k_h)^{-1}  (\partial \pi_h^k)^\top n_h - g^{-1}(\partial \pi )^\top n\right)}_{L^\infty}
	\abs{\sum_{\bar{e}_h\in \bar{\mathcal{E}}_h}\int_{ \bar{e}_h} \{\abs{(\nabla \bar{u}^{k,l}_h)^\top} \}[\bar{v}_h ]l_{g}dE }\\
	\leq &Ch^{k+1} \norm{\sqrt{\beta_h}^{-1}\nabla_g T_h^ku_h^{k,l}}_{L^2(\partial \Gamma_h)} \norm{\sqrt{\beta_h} v_h}_{L^2(\partial \Gamma_h)}\\
	=& \mathcal{O}(h^{k+1}) \norm{T_h^ku_h^{k,l}}_{H^1(\Gamma_h)}\norm{ v_h}_*=\mathcal{O}(h^{k+1}) \norm{T_h^ku_h^{k,l}}_{V_h^l}\norm{ v_h}_{V_h^l}.
	\end{aligned}
	\end{equation*}
	while the second last relation we used Cauchy--Schwarz inequality, Lemma \ref{lem:DG_trace}, the norm defined in \eqref{eq:trace_norm} and the relation that $\beta_h=\frac{\beta}{h}$. Similarly, for the third term
	\begin{equation*}
	\begin{aligned}
	&-  \sum_{\bar{e}_h\in \bar{\mathcal{E}}_h}\int_{ \bar{e}_h} \{(\nabla \bar{v}_h)^\top \left( (g^k_h)^{-1}  (\partial \pi_h^k)^\top n_h - g^{-1}(\partial \pi )^\top n\right) l_{g^k_h} \}[\bar{u}^{k,l}_h ]dE \\
	= &  \mathcal{O}(h^{k+1})  \norm{v_h}_{H^1(\Gamma_h)}\norm{ T_h^ku_h^{k,l}}_* =\mathcal{O}(h^{k+1})\norm{ v_h}_{V_h^l} \norm{T_h^ku_h^{k,l}}_{V_h^l}.
	\end{aligned}
	\end{equation*}
	The fourth and the fifth term are also quite similar, while we only have to take care of the geometric error contributed by $\norm{l_{g_h^k}-l_g}_{L^\infty}$, then we will end up with the same estimate of the second and the third ones up to different constants.
	The rest is quite trivial and we omit the details.	
}

%\bibliographystyle{plain}
%\bibliography{mybibfile}

\begin{thebibliography}{10}
        		
        		\bibitem{AbSt1964}
        		Abramowitz~M, Stegun~I~A.
        		\newblock {\em Handbook of mathematical functions with formulas, graphs, and
        			mathematical tables}, volume~55 of {\em National Bureau of Standards Applied
        			Mathematics Series}.
        		\newblock For sale by the Superintendent of Documents, U.S. Government Printing
        		Office, Washington, D.C., 1964.
        		
        		
        		\bibitem{ABCM2001}
        		Arnold~D, Brezzi~F, Cockburn~B,  Marini~L.
        		\newblock Unified analysis of discontinuous Galerkin methods for elliptic problems.
        		\newblock {\em SIAM J. Numer. Anal.},  39  (2001/02),  no. 5, 1749--1779.
        		
        		
        		\bibitem{AlSG2020}
        		Alliez~P, Saboret~L,  Guennebaud~G.
        		\newblock Poisson surface reconstruction.
        		\newblock In {\em {CGAL} User and Reference Manual}. {CGAL Editorial Board},
        		{5.1.1} edition, 2020.
        		
        		\bibitem{AntBufPer06}
        		Antonietti~P~F, Buffa~A, Perugia~I.
        		\newblock Discontinuous {G}alerkin approximation of the {L}aplace eigenproblem.
        		\newblock {\em Comput. Methods Appl. Mech. Engrg.}, 195:3483--3503, 2006.
        		
        		\bibitem{AntDedMadStaStiVer15}
        		Antonietti~P~F, Dedner~A, Madhavan~P, Stangalino~S, et al.
        		\newblock High order discontinuous {G}alerkin methods for elliptic problems on
        		surfaces.
        		\newblock {\em SIAM J. Numer. Anal.}, 53(2):1145--1171, 2015.
        		
        		\bibitem{Arn82}
        		Arnold~D~N.
        		\newblock An interior penalty finite element method with discontinuous
        		elements.
        		\newblock {\em SIAM J. Numer. Anal.}, 19:742--760, 1982.
        		
        		\bibitem{Aubin1982}
        		Aubin~T.
        		\newblock Best constants in the {S}obolev imbedding theorem: the {Y}amabe
        		problem.
        		\newblock In {\em Seminar on {D}ifferential {G}eometry}, volume 102 of {\em
        			Ann. of Math. Stud.}, pages 173--184. Princeton Univ. Press, Princeton, N.J.,
        		1982.
        		
        		\bibitem{BabOsb91}
        		Babu{\v{s}}ka~I, Osborn~J.
        		\newblock {\em Eigenvalue problems}, volume~II of {\em Handbook of Numerical
        			Analysis}.
        		\newblock North-Holland, Amsterdam, 1991.
        
        \bibitem{BaiLi23}
        Bai~G, Li~B.
        \newblock A new approach to the analysis of parametric finite element
          approximations to mean curvature flow.
        \newblock {\em Found. Comput. Math.}, pages 1--63, 2023.
        \newblock Online.
        		
        		\bibitem{BarDedQua15}
        		Bartezzaghi~A, Dedè~L, Quarteroni~A.
        		\newblock	Isogeometric Analysis of high order Partial Differential Equations on surfaces,
        			\newblock {\em Computer Methods in Applied Mechanics and Engineering},
        		Vol. 295: 446-469, 2015.
        		
        		\bibitem{Bof10}
        		Boffi~D.
        		\newblock Finite element approximation of eigenvalue problems.
        		\newblock {\em Acta Numerica}, 2010:1--120, 2010.
        		
        		\bibitem{BONITO20201}
        		Bonito~A, Demlow~A, Nochetto~R~H.
        		\newblock Finite element methods for the laplace–beltrami operator.
        		\newblock In {\em Geometric Partial Differential Equations - Part I}, volume~21
        		of {\em Handbook of Numerical Analysis}, pages 1 -- 103. Elsevier, 2020.
        		
        		\bibitem{BonDemOwe18}
        		Bonito~A, Demlow~A,  Owen~J.
        		\newblock A priori error estimates for finite element approximations to
        		eigenvalues and eigenfunctions of the {L}aplace--{B}eltrami operator.
        		\newblock {\em SIAM J. Numer. Anal.}, 56(5):2963--2988, 2018.
        
            \bibitem{Ci2002}
            Ciarlet~P~ G.
             \newblock {\em The finite element method for elliptic problems}.
           \newblock {vol. 40. of Classics in Applied Mathematics}, Reprint of the 1978 original [North-Holland, Amsterdam.
            \newblock Society for Industrial and Applied Mathematics (SIAM), Philadelphia, PA, 2002.
        
		
		\bibitem{DecDziEll05}
		Deckenlnick~K, Dziuk~G,  Elliott~C~M.
		\newblock Computation of geometric partial differential equations and mean
		curvature flow.
		\newblock {\em Acta Numer.}, 14:139--232, 2005.
		
		\bibitem{DednerMadhavanStinner2016}
		Dedner~A, Madhavan~P.
		\newblock Adaptive discontinuous {G}alerkin methods on surfaces.
		\newblock {\em Numer. Math.}, 132(2):369--398, 2016.
		
		\bibitem{DedMadSti13}
		Dedner~A, Madhavan~P, Stinner~B.
		\newblock Analysis of the discontinuous {G}alerkin method for elliptic problems
		on surfaces.
		\newblock {\em IMA J. Numer. Anal.}, 33(3):952--973, 2013.
		
		\bibitem{Demlow2009}
		Demlow~A.
		\newblock Higher-order finite element methods and pointwise error estimates for
		elliptic problems on surfaces.
		\newblock {\em SIAM J. Numer. Anal.}, 47(2):805--827, 2009.
		
		\bibitem{DiEr2012}
		Di~Pietro~D~A,   Ern~A.
		\newblock {\em Mathematical aspects of discontinuous {G}alerkin methods},
		volume~69 of {\em Math\'{e}matiques \& Applications (Berlin) [Mathematics \&
			Applications]}.
		\newblock Springer, Heidelberg, 2012.
		
		\bibitem{DonGuo20}
		Dong~G,  Guo~H.
		\newblock Parametric polynomial preserving recovery on manifolds.
		\newblock {\em SIAM J. Sci. Comput.}, 42(3):A1885--A1912, 2020.
		
		\bibitem{DonGuoGuo24}
		Dong~G, Guo,~H, Guo~T.		
        \newblock Superconvergence of differential structure for finite element methods on perturbed surface meshes.
		\newblock {\em J. Comput. Math.}, online first, pages 1--23, 2024.
		
		\bibitem{DuJu2005}
		Du~Q, Ju~L.
		\newblock Finite volume methods on spheres and spherical centroidal {V}oronoi
		meshes.
		\newblock {\em SIAM J. Numer. Anal.}, 43(4):1673--1692 (electronic), 2005.
		
\bibitem{DuanLi24}
Duan~B, Li~B.
\newblock New artificial tangential motions for parametric finite element
  approximation of surface evolution.
\newblock {\em SIAM J. Sci. Comput.}, 46(1):A587--A608, 2024.

		\bibitem{Dziuk1988}
		Dziuk~G.
		\newblock Finite elements for the {B}eltrami operator on arbitrary surfaces.
		\newblock In {\em Partial differential equations and calculus of variations},
		volume 1357 of {\em Lecture Notes in Math.}, pages 142--155. Springer,
		Berlin, 1988.
		
		\bibitem{DziukElliott07}
		Dziuk~G,   Elliott~C~M.
		\newblock Finite elements on evolving surfaces.
		\newblock {\em IMA J. Numer. Anal.}, 27:262--292, 2007.
		
		\bibitem{DziukElliott2013}
		Dziuk~G,   Elliott~C~M.
		\newblock Finite element methods for surface {PDE}s.
		\newblock {\em Acta Numer.}, 22:289--396, 2013.
		
		\bibitem{GraLehReu18}
		Grande~J,  Lehrenfeld~C,  Reusken~A.
		\newblock   Analysis of a high-order trace finite element method for PDEs on level set surfaces. 
		\newblock {\em SIAM J. Numer. Anal.}, 56(1):228-255, 2018.	
		
		
		\bibitem{Guo20}
		Guo~H.
		\newblock Surface {C}rouzeix-{R}aviart element for the {L}aplace-{B}eltrami
		equation.
		\newblock {\em Numer. Math.}, 144(3):527--551, 2020.
		
		
		
		\bibitem{HeWa2008}
		Hesthaven~J~S, Warburton~T.
		\newblock {\em Nodal discontinuous {G}alerkin methods}, volume~54 of {\em Texts
			in Applied Mathematics}.
		\newblock Springer, New York, 2008.
		\newblock Algorithms, analysis, and applications.
%		

\bibitem{HuLi22}
Hu~J,   Li~B.
\newblock Evolving finite element methods with an artificial tangential
  velocity for mean curvature flow and willmore flow.
\newblock {\em Numer. Math.}, 152(1):127--181, 2022.

		
		\bibitem{KaBH2006}
		Kazhdan~M, Bolitho~M,  Hoppe~H.
		\newblock {Poisson Surface Reconstruction}.
		\newblock In Alla Sheffer and Konrad Polthier, editors, {\em Symposium on
			Geometry Processing}. The Eurographics Association, 2006.
		\bibitem{SPR}
		Kazhdan~M, Hoppe~H, Screened poisson surface reconstruction. ACM Transactions on Graphics (ToG), 32(3), 1-13, 2013.


\bibitem{KovLiLub19}
Kov\'acs~B, Li~B, Lubich~C.
\newblock A convergent evolving finite element algorithm for mean curvature
  flow of closed surfaces.
\newblock {\em Numer. Math.}, 143(1):797--853, 2019.

		
		\bibitem{LiShi16}
		Li~Z,  Shi~Z.
		\newblock A convergent point integral method for isotropic elliptic equations
		on a point cloud.
		\newblock {\em Multiscale Modeling and Simulation}, 14(2):874--905, 2016.
		
		\bibitem{LaiZha17}
		Lai~R,  Zhao~H.
		\newblock Multi-scale Non-Rigid Point Cloud Registration Using Rotation-invariant Sliced-Wasserstein Distance via Laplace-Beltrami Eigenmap.
		\newblock {\em SIAM J. Imag. Sci.}, 10(2):449--483, 2017.
		
		\bibitem{LiaZha13}
		Liang~J,  Zhao~H.
		\newblock Solving partial differential equations on point clouds.
		\newblock {\em SIAM J. Sci. Comput.}, 35(3):A1461--A1486, 2013.
		
		\bibitem{GR}
		Lu~W, Shi~Z, Wang~B, Sun~J. 
		\newblock Surface Reconstruction Based on Modified Gauss Formula.
		\newblock {\em Transactions on Graphics}, 38(1):1-18, 2018.
		
		\bibitem{OlshanskiiReuskenGrande2009}
		Olshanskii~M~A, Reusken~A,  Grande~J.
		\newblock A finite element method for elliptic equations on surfaces.
		\newblock {\em SIAM J. Numer. Anal.}, 47(5):3339--3358, 2009.
		
		\bibitem{PanRabYan21}
		Pan~Q, Rabczuk~T,   Yang~X. 
		\newblock Subdivision-based isogeometric analysis for second order partial differential equations on surfaces. 
		\newblock {\em Comput Mech} 68, 1205–1221, 2021.
		
		\bibitem{Rivi2008}
		Rivi\'ere~B.
		\newblock {\em Discontinuous {G}alerkin methods for solving elliptic and
			parabolic equations}, volume~35 of {\em Frontiers in Applied Mathematics}.
		\newblock Society for Industrial and Applied Mathematics (SIAM), Philadelphia,
		PA, 2008.
		\newblock Theory and implementation.
		
		\bibitem{Wen05}
		Wendland~H.
		\newblock {\em Scattered Data Approximation}, volume~17 of {\em Cambridge
			Monographs on Applied and Computational Mathematics}.
		\newblock Cambridge University Press, first edition, 2005.
		
	\end{thebibliography}
\end{document}